\documentclass[11pt]{amsart}

\usepackage[utf8]{inputenc}
\usepackage[english]{babel}

\usepackage[dvipsnames]{xcolor}
\usepackage{environ}
\NewEnviron{mynote}{%
    \marginpar{\footnotemark}
    \footnotetext{{\color{black}\BODY}}}[\par]

\usepackage[margin=1in]{geometry}  
\usepackage{graphicx}              
\usepackage{amsmath}              
\usepackage{amsfonts}              
\usepackage{ marvosym } 

\usepackage[nomain,style=long3col,sort=use,section=section]{glossaries}

\usepackage{setspace}
\usepackage{float}

\usepackage{amsthm}                
\usepackage{amssymb}              
\usepackage{tikz}                       
\usetikzlibrary{intersections}
\usetikzlibrary{plotmarks}
\usetikzlibrary{shapes}
\usetikzlibrary{arrows}

\usepackage{listings} 
\usepackage{upgreek}
\usepackage{capt-of}
\usepackage{ stmaryrd }

\usepackage{caption}
\usepackage{url}
\usepackage{young}
\usepackage{ytableau}

\usepackage{color}
\usepackage[pdftex]{hyperref}
\hypersetup{
colorlinks=true,
 linkcolor={green!30!black},
    citecolor={blue!50!black},
    urlcolor={blue!80!black}
}

\theoremstyle{plain}

\newtheorem{thm}{Theorem}[section]
\newtheorem{lem}[thm]{Lemma}

\newtheorem{prop}[thm]{Proposition}
\newtheorem{cor}[thm]{Corollary}
\newtheorem*{theorem*}{Theorem}
\newtheorem*{cor*}{Corollary}

\theoremstyle{definition}

\newtheorem{rem}{Remark}
\newtheorem{ex}{Example}

\newtheorem{defn}[thm]{Definition}

\DeclareMathOperator{\Spec}{Spec}
\DeclareMathOperator{\cone}{Cone}

\DeclareMathOperator{\NW}{NW}

\DeclareMathOperator{\TV}{TV}

\DeclareMathOperator{\G}{G}

\DeclareMathOperator{\casei}{(i)}

\DeclareMathOperator{\dom}{dom}
\DeclareMathOperator{\Ess}{Ess}
\DeclareMathOperator{\Def}{Def}
\DeclareMathOperator{\Stab}{Stab}

\DeclareMathOperator{\ce}{\mathfrak{c}}
\DeclareMathOperator{\Val}{Val}
\DeclareMathOperator{\rank}{rank}

\newcommand{\RR}{\mathbb{R}}      
\newcommand{\QQ}{\mathbb{Q}}     
\newcommand{\ZZ}{\mathbb{Z}}      
\newcommand{\CC}{\mathbb{C}}

\newcommand{\Ione}{\mathcal{I}_G^{(1)}}

\begin{document}

\title{Rigid toric matrix Schubert varieties}

\author{Irem Portakal}

\address{Technical University of Munich\\ Department of Mathematics, Munich, Germany}

\email{mail@irem-portakal.de}


\begin{abstract}
For a given permutation $\pi \in S_N$, Fulton proves that the matrix Schubert variety $\overline{X_{\pi}} \cong Y_{\pi} \times \mathbb{C}^q$ can be defined via certain rank conditions encoded in the Rothe diagram of $\pi$. In the case where $Y_{\pi}:=\TV(\sigma_{\pi})$ is toric (with respect to a $(\mathbb{C}^*)^{2N-1}$ action), we show that it can be described as an edge ideal of a bipartite graph $G^{\pi}$. We characterize the lower dimensional faces of the associated so-called edge cone $\sigma_{\pi}$ explicitly in terms of subgraphs of $G^{\pi}$ and present a combinatorial study for the first order deformations of $Y_{\pi}$. We prove that $Y_{\pi}$ is rigid if and only if the three-dimensional faces of $\sigma_{\pi}$ are all simplicial. Moreover, we reformulate this result in terms of Rothe diagram of $\pi$.

\end{abstract}

\maketitle
\section{Introduction}
In this paper, we are studying the \emph{matrix Schubert varieties} $\overline{X_{\pi}} \cong Y_{\pi} \times \mathbb{C}^q$ associated to a permutation $\pi \in S_N$, where $q$ is maximal possible. These varieties first appear during Fulton’s study of the degeneracy loci of flagged vector bundles in \cite{Fulton92}. In \cite{factschu}, Knutson and Miller show that Schubert polynomials are multidegrees of matrix Schubert varieties. Moreover the matrix Schubert variety is in fact related to Schubert variety $X_{\pi}$ in the full flag manifold  via the isomorphism in (\cite{kazhdanlusztig}, Lemma A.4).  The matrix Schubert varieties are normal and one can define them by certain rank conditions encoded in the \emph{Rothe diagram}. Our goal is to investigate the natural restricted torus action on  these varieties. Escobar and M\'{e}sz\'{a}ros \cite{escobarmeszaros} study the \emph{toric matrix Schubert varieties} via understanding their moment polytope. We present a reformulation of their classification in terms of \emph{bipartite graphs}. The significance of this restatement is that it allows one to study the \emph{first order deformations} of the matrix Schubert variety in terms of graphs by \cite{rigidportakal}. The toric varieties arising from bipartite graphs have been studied in various papers \cite{herzog2015}, \cite{binomialbook}, \cite{hibiwalk}, \cite{valenciavillareal} in different perspectives. We will review few aspects of this theory and bring our attention to the classification of the \emph{rigid} toric matrix Schubert varieties. The toric varieties arising from graphs enable us to produce many interesting examples of rigid varieties. In fact, the first example of a rigid singularity in \cite{firstrigidpaper} is the cone over the Segre embedding $\mathbb{P}^2 \times \mathbb{P}^1 \rightarrow \mathbb{P}^{2r+1}$ which is the affine toric variety associated to the complete bipartite graph $K_{r+1, 2}$. Following the techniques in \cite{altmann96,rigidportakal} for the study of deformations of toric varieties, we classify rigid toric varieties $Y_{\pi}$ in terms of bipartite graphs and Rothe diagram.\\

\noindent The organization of the paper is as follows. In preliminaries we present some basic facts on matrix Schubert varieties and give a brief exposition of toric varieties arising from bipartite graphs. In Section \ref{torusactiononmsv} we reformulate the question of classification of toric matrix Schubert varieties to bipartite graphs. We then indicate how graphs may be used to investigate the complexity of the torus action in the sense of $T$-varieties \cite{pdivisor}, \cite{geometryofT}. Section \ref{section4} starts with a discussion of deformation theory of toric varieties. Furthermore it provides a detailed exposition of faces of the moment (edge) cone of $Y_{\pi}$. In Lemma \ref{schulemma1}, we characterize the extremal rays of the edge cone. In Proposition \ref{twofacesschu} and Proposition \ref{threefacesschu}, we present conditions for extremal rays to span a two-dimensional face and a three-dimensional face respectively. Finally, we conclude the following result.
\begin{theorem*}[{Theorem \ref{toricschurigid}}]
The toric variety $Y_{\pi}$ is rigid if and only if its moment (edge) cone has simplicial three-dimensional faces.
\end{theorem*}
\noindent We translate this result to the Rothe diagram of $\pi$ and restate the classification in terms of certain shapes on the diagram, in particular solely depending on the pattern of the {\em essential set}  (Definition \ref{def: essential set}) of $\pi$.
\begin{cor*}[Corollary \ref{reformulation}]
Let $\Ess(\pi)= \{ (x_i,y_i) \ | \ x_{k+1} < \ldots <x_{1} \text{ and } y_{1} < \ldots < y_{k+1} \}$ with $k \geq 3$ be the essential set of the Rothe diagram of $\pi \in S_N$. Then the toric variety $Y_{\pi}$ is rigid if and only if 
\begin{itemize}
\item[$\bullet$] $(x_1,y_1) \neq (m,2) $ and $(x_{k+1}, y_{k+1}) \neq (2,n)$ or
\item[$\bullet$] for any $i \in [k], $ $(x_{i}, y_i) \neq (x_{i+1} +1 , y_{i+1}-1) $.
\end{itemize} 
where $m$ is the length and $n$ is the width of the smallest rectangle containing $L(\pi)$ from Definition~\ref{def: essential set}.
\end{cor*}

\section{Preliminaries}\label{preliminaries}

\subsection{Matrix Schubert varieties}
\noindent In this section, we adopt the conventions from \cite{escobarmeszaros} for matrix Schubert varieties.  We are mainly interested in matrix Schubert varieties for their effective torus actions and deformations. The statements presented in this section can be found in \cite{Fulton92} and \cite{factschu}. \\

\noindent Let $\pi \in S_N$ be a permutation. We denote its {\it permutation matrix} by $\pi \in \CC^{N \times N}$ as well and define it as follows:
  \[  \pi_{(i,j)}=\left\{
                \begin{array}{ll}
                  1, \ \text{if } \pi(j)=i \\
                  0, \text{ otherwise}.\\
                \end{array}
              \right.  \]

\vspace{0.2cm}
\noindent Let $B_{\_}$ denote the group of invertible lower triangular matrices and $B_{+}$ denote the group of invertible upper triangular $N\times N$ matrices. The product $B_{\_} \times B_{+}$ acts from left on $\CC^{N \times N}$ and its action defined as:
\begin{eqnarray*}
(B_{\_} \times B_{+}) \times \CC^{N \times N} &\longrightarrow &\CC^{N \times N}\\
((M_{\_},M_{+}) ,\mathcal{M}) &\mapsto&  M_{\_} \mathcal{M} M_{+}^{-1}
\end{eqnarray*}

\begin{defn}
Let $\mathcal{M}_{(a,b)} \in \CC^{a \times b}$ be the $a \times b$ matrix located at the upper left corner of $\mathcal{M}~\in~\CC^{N \times N}$, where $1\leq a \leq N$ and $1 \leq b \leq N$. The rank function of $\mathcal{M}$ is defined as $r_{\mathcal{M}}(a,b):=\text{rank}(\mathcal{M}_{(a,b)})$. 
\end{defn}
\vspace{0.2cm}
\noindent Note that the multiplication of a matrix $\mathcal{M} \in \CC^{N \times N}$ on the left with $M_{\_}$ corresponds to the downwards row operations and multiplication of $\mathcal{M}$ on the right with $M_{+}$ corresponds to the rightward column operations. Hence, one observes that $\mathcal{M} \in B_{\_} \pi B_{+}$ if and only if $r_{\mathcal{M}}({a,b})~=~ r_{\pi} (a,b)$ for all $(a,b) \in [N] \times [N]$.
 \vspace{0.2cm}

\begin{defn}
The Zariski closure of the orbit $\overline{X_{\pi}}:=\overline{B_{\_} \pi B_{+}} \subseteq \CC^{N \times N}$ is called \textit{the matrix Schubert variety} of $\pi$. 
\end{defn}

\noindent Rothe presented a combinatorial technique for visualizing inversions of the permutation $\pi$.

\begin{defn}
The \it{Rothe diagram} of $\pi$ is defined as $D(\pi) = \{ (\pi(j),i): i<j, \pi(i) > \pi(j)\}$.
\end{defn}

\noindent One can draw the Rothe diagram $D(\pi)$ in the following way: Consider the permutation matrix $\pi$ in an $N \times N$ grid. Cross out each box containing 1 and all the other boxes to the south and east of each box containing 1.

\begin{figure}[H]
\centering
\begin{tikzpicture}[scale=0.625,every path/.style={>=latex},every node/.style={draw,circle,fill=black,scale=0.2}]
 \path[draw, fill=cyan!40,opacity=1] (0,3)--(1,3)--(1,4)--(0,4)--cycle;
 \path[draw, fill=cyan!40,opacity=1] (2,1)--(3,1)--(3,2)--(2,2)--cycle;

  \node            (a0) at (0,0)  {};
  \node            (b0) at (1,0)  {};
  \node            (c0) at (2,0) {};
  \node            (d0) at (3,0) {};
  \node            (e0) at (4,0) {};

  \node            (a1) at (0,1)  {};
  \node            (b1) at (1,1)  {};
  \node            (c1) at (2,1) {};
  \node            (d1) at (3,1) {};
  \node            (e1) at (4,1) {};

  \node            (a2) at (0,2)  {};
  \node            (b2) at (1,2)  {};
  \node            (c2) at (2,2) {};
  \node            (d2) at (3,2) {};
  \node            (e2) at (4,2) {};

  \node            (a3) at (0,3)  {};
  \node            (b3) at (1,3)  {};
  \node            (c3) at (2,3) {};
  \node            (d3) at (3,3) {};
  \node            (e3) at (4,3) {};

  \node            (a4) at (0,4)  {};
  \node            (b4) at (1,4)  {};
  \node            (c4) at (2,4) {};
  \node            (d4) at (3,4) {};
  \node            (e4) at (4,4) {};

  \node            (p1) at (0.5,2.5) {};
  \node            (p2) at (1.5,3.5) {};
  \node            (p3) at (2.5,0.5) {};
  \node            (p4) at (3.5,1.5) {};

\draw[-,red] (p1) edge (0.5,0);
\draw[-,red] (p1) edge (4,2.5);

\draw[-,red] (p2) edge (1.5,0);
\draw[-,red] (p2) edge (4,3.5);

\draw[-,red] (p3) edge (2.5,0);
\draw[-,red] (p3) edge (4,0.5);

\draw[-,red] (p4) edge (3.5,0);
\draw[-,red] (p4) edge (4,1.5);

 \draw[-] (a0) edge (e0);
  \draw[-] (a1) edge (e1);
   \draw[-] (a2) edge (e2);
    \draw[-] (a3) edge (e3);
    \draw[-] (a4) edge (e4);

 \draw[-] (a0) edge (a4); 
 \draw[-] (b0) edge (b4); 
 \draw[-] (c0) edge (c4); 
 \draw[-] (d0) edge (d4); 
 \draw[-] (e0) edge (e4); 

\end{tikzpicture}\captionof{figure}{The Rothe Diagram of $[2143] \in S_4$.}
\end{figure}
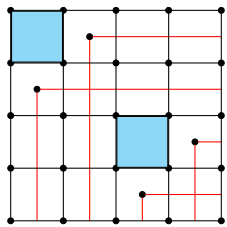

\begin{defn}\label{def: essential set}
The connected part containing the box $(1,1)$ in the diagram is called the dominant piece and denoted by $\dom(\pi)$. The set consisting of the south-east corners of each connected component of $D(\pi)$ is called the essential set and denoted as $\Ess(\pi)$. Let $\NW(\pi)$ denote the union of the boxes located to the north-west of each box in $D(\pi)$. Finally, let $L(\pi) := \NW(\pi) \setminus \dom(\pi)$ and $L'(\pi):= L(\pi) \setminus D(\pi)$. 
\end{defn}

\noindent In Figure \ref{domrep} below, one can visualize these sets in the Rothe diagram of $[2143]\in~ S_4$.
\vspace{0.2cm}
\begin{center}

\begin{tikzpicture}[baseline=1cm,scale=0.625,every path/.style={>=latex},every node/.style={draw,circle,fill=black,scale=0.2}]

 \path[draw, fill=cyan!40,opacity=1] (0,3)--(1,3)--(1,4)--(0,4)--cycle;
  \node            (a0) at (0,0)  {};
  \node            (b0) at (1,0)  {};
  \node            (c0) at (2,0) {};
  \node            (d0) at (3,0) {};
  \node            (e0) at (4,0) {};

  \node            (a1) at (0,1)  {};
  \node            (b1) at (1,1)  {};
  \node            (c1) at (2,1) {};
  \node            (d1) at (3,1) {};
  \node            (e1) at (4,1) {};

 \ node            (a2) at (0,2)  {};
  \node            (b2) at (1,2)  {};
  \node            (c2) at (2,2) {};
  \node            (d2) at (3,2) {};
  \node            (e2) at (4,2) {};

  \node            (a3) at (0,3)  {};
  \node            (b3) at (1,3)  {};
  \node            (c3) at (2,3) {};
  \node            (d3) at (3,3) {};
  \node            (e3) at (4,3) {};

  \node            (a4) at (0,4)  {};
  \node            (b4) at (1,4)  {};
  \node            (c4) at (2,4) {};
  \node            (d4) at (3,4) {};
  \node            (e4) at (4,4) {};

 \draw[-] (a0) edge (e0);
  \draw[-] (a1) edge (e1);
   \draw[-] (a2) edge (e2);
    \draw[-] (a3) edge (e3);
    \draw[-] (a4) edge (e4);

 \draw[-] (a0) edge (a4); 
 \draw[-] (b0) edge (b4); 
 \draw[-] (c0) edge (c4); 
 \draw[-] (d0) edge (d4); 
 \draw[-] (e0) edge (e4); 

\end{tikzpicture}
\hspace{1cm}
\begin{tikzpicture}[baseline=1cm,scale=0.625,every path/.style={>=latex},every node/.style={draw,circle,fill=black,scale=0.2}]

 \path[draw, fill=BurntOrange!40,opacity=1] (0,1)--(3,1)--(3,4)--(0,4)--cycle;
  \node            (a0) at (0,0)  {};
  \node            (b0) at (1,0)  {};
  \node            (c0) at (2,0) {};
  \node            (d0) at (3,0) {};
  \node            (e0) at (4,0) {};

  \node            (a1) at (0,1)  {};
  \node            (b1) at (1,1)  {};
  \node            (c1) at (2,1) {};
  \node            (d1) at (3,1) {};
  \node            (e1) at (4,1) {};

  \node            (a2) at (0,2)  {};
  \node            (b2) at (1,2)  {};
  \node            (c2) at (2,2) {};
  \node            (d2) at (3,2) {};
  \node            (e2) at (4,2) {};

  \node            (a3) at (0,3)  {};
  \node            (b3) at (1,3)  {};
  \node            (c3) at (2,3) {};
  \node            (d3) at (3,3) {};
  \node            (e3) at (4,3) {};

  \node            (a4) at (0,4)  {};
  \node            (b4) at (1,4)  {};
  \node            (c4) at (2,4) {};
  \node            (d4) at (3,4) {};
  \node            (e4) at (4,4) {};

 \draw[-] (a0) edge (e0);
  \draw[-] (a1) edge (e1);
   \draw[-] (a2) edge (e2);
    \draw[-] (a3) edge (e3);
    \draw[-] (a4) edge (e4);

 \draw[-] (a0) edge (a4); 
 \draw[-] (b0) edge (b4); 
 \draw[-] (c0) edge (c4); 
 \draw[-] (d0) edge (d4); 
 \draw[-] (e0) edge (e4); 

\end{tikzpicture}
\hspace{1cm}
\begin{tikzpicture}[baseline=1cm,scale=0.625,every path/.style={>=latex},every node/.style={draw,circle,fill=black,scale=0.2}]
 \path[draw, fill=cyan!40,opacity=1] (0,3)--(1,3)--(1,4)--(0,4)--cycle;
 \path[draw, fill=cyan!40,opacity=1] (2,1)--(3,1)--(3,2)--(2,2)--cycle;
  \node            (a0) at (0,0)  {};
  \node            (b0) at (1,0)  {};
  \node            (c0) at (2,0) {};
  \node            (d0) at (3,0) {};
  \node            (e0) at (4,0) {};

  \node            (a1) at (0,1)  {};
  \node            (b1) at (1,1)  {};
  \node            (c1) at (2,1) {};
  \node            (d1) at (3,1) {};
  \node            (e1) at (4,1) {};

  \node            (a2) at (0,2)  {};
  \node            (b2) at (1,2)  {};
  \node            (c2) at (2,2) {};
  \node            (d2) at (3,2) {};
  \node            (e2) at (4,2) {};

  \node            (a3) at (0,3)  {};
  \node            (b3) at (1,3)  {};
  \node            (c3) at (2,3) {};
  \node            (d3) at (3,3) {};
  \node            (e3) at (4,3) {};

  \node            (a4) at (0,4)  {};
  \node            (b4) at (1,4)  {};
  \node            (c4) at (2,4) {};
  \node            (d4) at (3,4) {};
  \node            (e4) at (4,4) {};

 \draw[-] (a0) edge (e0);
  \draw[-] (a1) edge (e1);
   \draw[-] (a2) edge (e2);
    \draw[-] (a3) edge (e3);
    \draw[-] (a4) edge (e4);

 \draw[-] (a0) edge (a4); 
 \draw[-] (b0) edge (b4); 
 \draw[-] (c0) edge (c4); 
 \draw[-] (d0) edge (d4); 
 \draw[-] (e0) edge (e4); 

\end{tikzpicture}
\hspace{1cm}
\begin{tikzpicture}[baseline=1cm,scale=0.625,every path/.style={>=latex},every node/.style={draw,circle,fill=black,scale=0.2}]
 \path[draw, fill=Orchid!40,opacity=1] (0,1)--(2,1)--(2,2)--(3,2)--(3,4)--(1,4)--(1,3)--(0,3)--cycle;
  \node            (a0) at (0,0)  {};
  \node            (b0) at (1,0)  {};
  \node            (c0) at (2,0) {};
  \node            (d0) at (3,0) {};
  \node            (e0) at (4,0) {};

  \node            (a1) at (0,1)  {};
  \node            (b1) at (1,1)  {};
  \node            (c1) at (2,1) {};
  \node            (d1) at (3,1) {};
  \node            (e1) at (4,1) {};

  \node            (a2) at (0,2)  {};
  \node            (b2) at (1,2)  {};
  \node            (c2) at (2,2) {};
  \node            (d2) at (3,2) {};
  \node            (e2) at (4,2) {};

  \node            (a3) at (0,3)  {};
  \node            (b3) at (1,3)  {};
  \node            (c3) at (2,3) {};
  \node            (d3) at (3,3) {};
  \node            (e3) at (4,3) {};

  \node            (a4) at (0,4)  {};
  \node            (b4) at (1,4)  {};
  \node            (c4) at (2,4) {};
  \node            (d4) at (3,4) {};
  \node            (e4) at (4,4) {};

 \draw[-] (a0) edge (e0);
  \draw[-] (a1) edge (e1);
   \draw[-] (a2) edge (e2);
    \draw[-] (a3) edge (e3);
    \draw[-] (a4) edge (e4);

 \draw[-] (a0) edge (a4); 
 \draw[-] (b0) edge (b4); 
 \draw[-] (c0) edge (c4); 
 \draw[-] (d0) edge (d4); 
 \draw[-] (e0) edge (e4); 

\end{tikzpicture}


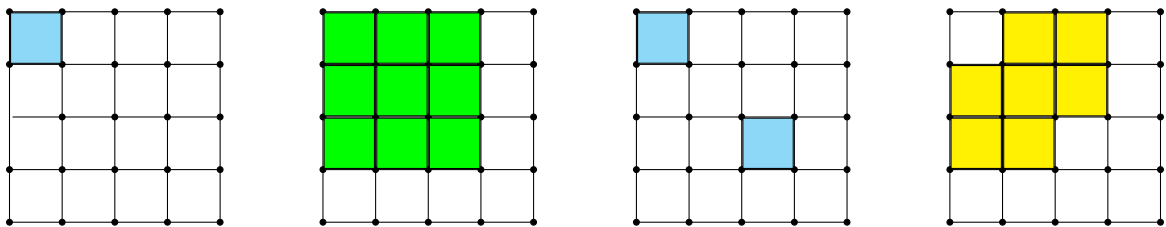
\captionof{figure}{The representations of $\dom(\pi)$, $\NW(\pi)$, $\Ess(\pi)$, and $L'(\pi)$. }\label{domrep}
\end{center}

\begin{thm}[{\cite[Proposition 3.3, Lemma 3.10]{Fulton92}}] \label{fultonschubert}
The matrix Schubert variety $\overline{X_{\pi}}$ is an affine variety of dimension $N^2-|D(\pi)|$. It can be defined as a scheme by the equations $r_{\mathcal{M}}({a,b}) \leq r_{\pi} (a,b)$ for all $(a,b) \in \Ess(\pi)$. 
\end{thm}

\begin{rem} \label{dimensionYpi} By the previous theorem, we observe that there exists no rank conditions imposed on the boxes which are not in $\NW(\pi)$. Thus these boxes are free in $\overline{X_{\pi}}$. Let $V_{\pi} \cong \mathbb{C}^{N^2 - |\NW(\pi)|}$ be the projection of the matrix Schubert variety $\overline{X_{\pi}} \subseteq \CC^{N \times N}$ onto these free boxes. Also, we define $Y_{\pi}$ as the projection onto the boxes of $L(\pi)$.  Note that one obtains $(a,b) \in \text{dom} (\pi)$ if and only if $r_\pi(a,b)=0$. Hence, $\overline{X_{\pi}} = Y_{\pi} \times V_{\pi}$ holds. In particular, by Theorem \ref{fultonschubert}, $$\dim (Y_{\pi}) = N^2- |D(\pi)| - N^2 - |\NW(\pi)| =|\NW(\pi)| - |D(\pi)| = |L'(\pi)|.$$
\end{rem}

\begin{ex}\label{firstexample}
The essential set for the permutation $\pi = [21 43] \in S_4$ consists of the boxes $(1,1)$ and $(3,3)$. Let $\mathcal{M}=(m_{ij}) \in \CC^{4 \times 4}$. First we note that $m_{11} =0$ since $(1,1) \in \dom(\pi)$. For the boxes in $L(\pi)$ one obtains the following inequality by Theorem \ref{fultonschubert}: \[ r_{\mathcal{M}}(3,3) = \rank(\mathcal{M}_{(3,3)}) =  \rank\left(
\begin{bmatrix}
    0       & m_{12} & m_{13} \\
    m_{21}       & m_{22} & m_{23} \\
    m_{31}       & m_{32} & m_{33} 
\end{bmatrix}
\right)
\leq 2.
\]
One obtains the ideal as generated by $I:=\det (\mathcal{M}_{(3,3)})$. In particular $\overline{X_{\pi}} \cong \mathbb{V}(I) \times \mathbb{C}^{7}$ and $\dim(Y_{\pi})=|L'(\pi)|=7$.
\end{ex}

\subsection{Edge cones of bipartite graphs}

In this section, we briefly introduce the construction of the toric varieties related to bipartite graphs as in \cite{binomialbook,rigidportakal}. We refer the reader to \cite{toricvarieties} for details on the toric varieties and the notations. In particular $N \cong \mathbb{Z}^n$ stands for a lattice and $M=\text{Hom}_{\mathbb{Z}}(N, \mathbb{Z})$ is its dual lattice. We denote their associated vector spaces as $N_{\mathbb{Q}}:=N \otimes_{\mathbb{Z}} \mathbb{Q}$ and $M_{\mathbb{Q}}:=M \otimes_{\mathbb{Z}} \mathbb{Q}$. \\

\noindent Let $G \subseteq K_{m,n}$ be a bipartite graph with edge set $E(G)$ and vertex set $V(G)$. One defines the edge ring as $$\text{Edr}(G):= \mathbb{C}[t_i t_j \ | \ (i,j) \in E(G)].$$ Consider the following morphism:
$$\varphi_G: \mathbb{C}[x_1,...,x_{|E(G)|}]\longrightarrow \text{Edr}(G)$$
$$x_{e} \mapsto t_i t_j \text{ with } e=(i,j).$$  

\noindent The kernel of this map is a toric ideal and it is called the \emph{edge ideal} of $G$. The affine normal \emph{toric variety associated to bipartite graph $G$} is $$\TV(G):= \Spec(\mathbb{C}[x_1,...,x_{|E(G)|}] \slash \ker \varphi_G ).$$ 

\noindent Let $e_{i}$ denote a canonical basis element of $\mathbb{Z}^{m} \times 0$ for $i=1, \ldots, m$ and  $f_{j}$ denote a canonical basis element of $0 \times \ZZ^n$ for $j=1,\ldots,n$. Set the lattices for the associated cones of the toric variety $\TV(G)$ as
\[N:= \ZZ^{m+n}/\overline{(1,-1)} \text{ and } M:=\ZZ^{m+n} \cap \overline{(1,-1)}^{\bot}\]
where $\overline{(1,-1)}:=\big\langle \sum_{i=1}^m e_i - \sum_{j=1}^n f_j \big\rangle$. We denote their associated vector spaces as $N_{\QQ}$ and $M_{\QQ}$.  In order to distinguish the elements of these vector spaces, we denote the ones in $N_{\QQ}$ by normal brackets and the ones in $M_{\QQ}$ by square brackets. For the same reason, we denote the canonical basis elements as $e_i \in N_{\QQ}$ and $e^i \in M_{\QQ}$. \\

\noindent Hence we obtain that the (dual) edge cone associated to $\TV(G)$ is $$\sigma_G^{\vee}=\cone(e^i + f^j \ | \ (i,j)\in E(G)) \subseteq M_{\QQ},$$ i.e.\ we have that $$\TV(G)= \Spec(\mathbb{C}[\sigma_G^{\vee} \cap M]).$$
We observe in Section \ref{torusactiononmsv} that the dual edge cone $\sigma_G^{\vee}$ is in fact isomorphic to the moment cone of a matrix Schubert variety. We use this fact in order to determine the complexity of the torus action on a matrix Schubert variety.

\begin{prop}[{\cite[Proposition 2.1, Lemma 2.17]{rigidportakal}}]\label{dimensionprop}
Let $G\subseteq K_{m,n}$ be a bipartite graph with $k$ connected components and $m+n$ vertices. Then the dimension of $\sigma_G^{\vee} \subset M_{\mathbb{R}}$ is $m+n-k$.
\end{prop}

\noindent Our aim is to study the first order deformations $T^1_{\TV(G)}$ of the affine toric variety $\TV(G)$ by \cite{altmann96}. One can describe the $T^1_{\TV(G)}$ via understanding the two and three-dimensional faces of the edge cone $\sigma_G \subseteq N_{\QQ}$. We explain this technique briefly in Section \ref{deformationtheory}. We first introduce terminology and notation from graph theory  to describe the rays and faces of $\sigma_G$ in terms of subgraphs of $G$. In Section \ref{rigiditysection}, we will describe the rigidity of $\TV(G)$ in terms of graphs. 

\begin{rem}\label{notconnected}
 Note that if the bipartite graph $G$ is the disjoint union of two connected bipartite graph $G= G_1 \sqcup G_2$, then we have $\TV(G) = \TV(G_1) \times \TV(G_2)$. Thus for the remainder of this section,  we assume that $G \subseteq K_{m,n}$ is connected.
\end{rem}

\begin{defn}
A nonempty subset $A$ of $V(G)$ is called an \textit{independent set} if it contains no adjacent vertices. An independent set $A\subsetneq V(G)$ is called a \textit{maximal independent set} if there is no other independent set containing it. Let $U_1$ and $U_2$ be the disjoint sets of vertices of $G$. We say that an independent set is \textit{one-sided} if it is contained either in $U_1$ or  $U_2$. In a similar way, $A=A_1 \sqcup A_2$ is called a \textit{two-sided independent set} if $\emptyset \neq A_1 \subsetneq U_1$ and $ \emptyset \neq A_2 \subsetneq U_2$. 
\end{defn}

\begin{defn} 
The \textit{neighbor set} of $A \subseteq V(G)$ is defined as $$N(A) : = \{v \in V(G) \ | \ v \text{ is adjacent to some vertex in } A\}.$$ The \textit{supporting hyperplane} of the dual edge cone $\sigma^{\vee}_G \subseteq M_{\QQ}$ associated to an independent set $\emptyset \neq A $ is defined as $$\mathcal{H}_A : = \{x \in M_{\mathbb{Q}} \ | \ \sum_{v_i \in A} x_i = \sum_{v_i \in N(A)} x_i\}.$$

\end{defn}
\noindent Note that since no pair of vertices of an independent set $A$ is adjacent, we obtain that $A \cap N(A) =\emptyset$.
\begin{defn}
\begin{enumerate}
\item[]
\item  A subgraph of $G$ with the same vertex set as $G$ is called a \textit{spanning subgraph} of $G$.
\item Let $A\subseteq V(G)$ be a subset of the vertex set of $G$. The \textit{induced subgraph} of $A$ is defined as the subgraph of $G$ formed from the vertices of $A$ and all of the edges connecting pairs of these vertices. We denote it as $\G[A]$ and we adopt the convention $G[\emptyset]=\emptyset$.
\end{enumerate}
\end{defn}

\noindent Now, we characterize the independent sets resulting a facet of $\sigma_G^{\vee}$. 

\begin{defn}\label{assocgraph}
Let $G[[A]]$ denote the subgraph of $G$ associated to the independent set $A$ and defined as 
$$
  \left\{
                \begin{array}{ll}
                G[A \sqcup N(A)] \sqcup G[(U_1\backslash A) \sqcup (U_2\backslash N(A))], \text{if} A \subseteq U_1 \text{ is one-sided.}\\
                G[A \sqcup N(A)] \sqcup G[(U_2\backslash A) \sqcup (U_1\backslash N(A))], \text{if }A\subseteq U_2\text{ is one-sided.} \\
               G[A_1 \sqcup N(A_1)] \sqcup G[A_2 \sqcup N(A_2)], \text{if }A=A_1\sqcup A_2 \text{ is two-sided.} 
    
                \end{array}
              \right.  $$

\noindent We define the \emph{associated bipartite subgraph} $\G \{A\}  \subseteq G$ to the independent set $A$ as the spanning subgraph $G[[A]] \sqcup \big(V(G) \backslash V(G[[A]])\big)$.
\end{defn}

\noindent Finally, we define the \textit{first independent sets} $\mathcal{I}_G^{(1)}$ of $G$ as\\
 \[
   \mathcal{I}_G^{(1)}:= \left\{\begin{array}{lr}
      \text{Two-sided maximal independent sets and one-sided independent sets } U_i \backslash \{\bullet\}\\
 \text{ where their associated bipartite subgraph has two connected components.}
        \end{array}\right\} 
  \]
Note that Definition \ref{assocgraph} becomes less technical for first independent sets by \cite[Proposition 2.9, Lemma 2.10]{rigidportakal}. Namely we obtain:\\

$\G\{U_i \backslash \{\bullet\}\} = \G[U_i \backslash \{\bullet\} \sqcup U_j] \sqcup \{\bullet\}$, for a one-sided first independent set with $i \neq j$ and \\ 

$\G\{A\} =  G[A_1 \sqcup N(A_1)] \sqcup G[A_2 \sqcup N(A_2)]$, for a two-sided first independent set $A=A_1\sqcup A_2$.\\

\noindent Denote the set of extremal ray generators (i.e.\ 1-dimensional faces) of $\sigma_G$ by  $\sigma_G^{(1)}$. Recall that there is a bijective inclusion-reversing correspondence between the faces of $\sigma_G$ and the faces of $\sigma_G^{\vee}$. Given a face $\tau \preceq \sigma_G^{\vee}$, we define the dual face $\tau^{*}$ of $\tau$ as $\{ x \in \sigma_G^{\vee} \ | \ \langle x, u \rangle = 0 \text{ for all } u \in \tau\}$. In particular the facets of $\sigma_G^{\vee}$ are in bijection with the extremal rays of $\sigma_G$.
\begin{thm}[{\cite[Theorem 2.8]{rigidportakal}}]\label{11thm}
There is a one-to-one correspondence between the set of extremal generators $\sigma_G^{(1)}$ and the first independent set $\mathcal{I}_G^{(1)}$. In particular, the map is given as
\begin{eqnarray*}
 \Pi \colon \mathcal{I}_G^{(1)} &\longrightarrow &\sigma_G^{(1)} \\
A &\mapsto&  \mathfrak{a} := (\mathcal{H}_{A_i} \cap \sigma_G^{\vee})^{*}
\end{eqnarray*}
for a fixed $i \in \{1,2\}$ with $A_i \neq \emptyset$.
\end{thm}
\begin{ex} We consider the bipartite graph $G \subset K_{2,2}$ obtained by removing one edge from the complete bipartite graph. The first independent set $\mathcal{I}_G^{(1)}$ for the graph $G$ is colored in green. The sets $\{1\}$ and $\{3\}$ are not in $\mathcal{I}_G^{(1)}$ since they are contained in the two-sided maximal independent set $\{1,3\}$. Hence their associated subgraph has three connected components. The cone $\sigma_G \subseteq N_{\QQ}$ is generated by $(1,0,0,0)$, $(0,0,1,0)$ and $(0,1,-1,0)$ corresponding respectively to the associated subgraphs.
\end{ex}

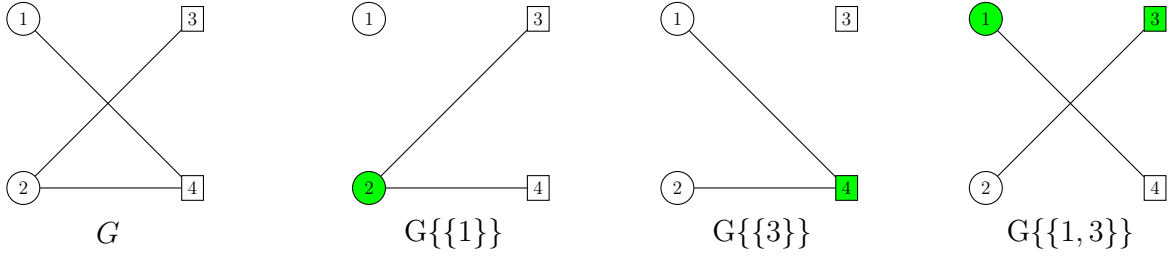
\begin{center}
\colorbox{white}{
\begin{minipage}[b]{0.27\linewidth}
\begin{tikzpicture}[baseline=1,scale=1.5,every path/.style={>=latex},every node/.style={draw,circle,fill=white,scale=0.6}]
  \node            (a) at (0,0)  {2};
  \node       [rectangle]     (b) at (1.5,0)  {4};
  \node            (c) at (0,1.5) {1};
  \node       [rectangle]     (d) at (1.5,1.5) {3};
\node [draw =none, fill=none,scale=1.6](e) at (0.75, -0.4) {$G$};

    \draw[-] (c) edge (b);
    \draw[-] (b) edge (a);
    \draw[-] (a) edge (d);
 
\end{tikzpicture}
\end{minipage}

\begin{minipage}[b]{0.24\linewidth}
\begin{tikzpicture}[baseline=1, scale=1.5,every path/.style={>=latex},every node/.style={draw,circle,fill=white,scale=0.6}]
  \node     [fill=green]          (a) at (0,0)  {2};
  \node        [rectangle]    (b) at (1.5,0)  {4};
  \node            (c) at (0,1.5) {1};
  \node       [rectangle]     (d) at (1.5,1.5) {3};
\node [draw =none, fill=none,scale=1.6](e) at (0.75, -0.4) {$\G\{ \{1\} \}$};

    \draw[-] (b) edge (a);
    \draw[-] (a) edge (d);

\end{tikzpicture}
\end{minipage}

\begin{minipage}[b]{0.24\linewidth}
\begin{tikzpicture}[baseline=1,scale=1.5,every path/.style={>=latex},every node/.style={draw,circle,fill=white,scale=0.6}]
  \node            (a) at (0,0)  {2};
  \node       [fill=green,rectangle]        (b) at (1.5,0)  {4};
  \node         (c) at (0,1.5) {1};
  \node      [rectangle]      (d) at (1.5,1.5) {3};
\node [draw =none, fill=none,scale=1.6](e) at (0.75, -0.4) {$\G\{\{3\} \}$};

    \draw[-] (c) edge (b);
    \draw[-] (b) edge (a);
 
\end{tikzpicture}
\end{minipage}

\begin{minipage}[b]{2.4\linewidth}
\begin{tikzpicture}[baseline=1,scale=1.5,every path/.style={>=latex},every node/.style={draw,circle,fill=white,scale=0.6}]
  \node            (a) at (0,0)  {2};
  \node           [rectangle] (b) at (1.5,0)  {4};
  \node       [fill=green]        (c) at (0,1.5) {1};
  \node       [fill=green,rectangle]        (d) at (1.5,1.5) {3};
\node [draw =none, fill=none,scale=1.6](e) at (0.75, -0.4) {$\G\{\{1,3\}\}$ };

    \draw[-] (c) edge (b);
    \draw[-] (a) edge (d);

\end{tikzpicture}
\end{minipage}
 }
\end{center}

\noindent The next result classifies $d$-dimensional faces of $\sigma_G$ via intersecting associated subgraphs related to first independent sets.

\begin{thm}[{\cite[Theorem 2.18]{rigidportakal}}]\label{facetheorem}
Let $ S \subseteq \mathcal{I}_G^{(1)}$ be a subset of $d$ first independent sets and let $\Pi$ be the bijection from Theorem \ref{11thm}. The extremal ray generators $\Pi(S)$ span a face of dimension $d$ if and only if the dimension of the dual edge cone of the spanning subgraph $\G[S]:=\bigcap_{A \in S} \G\{A\}$ is $m+n-d-1$, i.e.\ $\G[S]$ has $d+1$ connected components. In particular, the face is equal to $H_{\Val_S} \cap \sigma_G$ where $\Val_S$ is the degree sequence of the graph $\G[S]$ and $H_{\Val_S} = \{x \in N_{\QQ} \ | \ \langle \Val_S,x \rangle = 0 \}$ is the usual supporting hyperplane in $N_{\mathbb{Q}}$.
\end{thm}

\begin{ex}
All pairs of extremal rays of $\sigma_G$ generates a two-dimensional face of $\sigma_G$ since the intersection of all pairs of associated subgraphs has three connected components. In particular, the two-dimensional face generated by $(1,0,0,0)$ and $(0,1,-1,0)$, i.e.\ the edge cone of $\G\{\{1\}\} \cap~\G\{\{1,3\}\}$, is equal to $H_{[0,1,1,0]} \cap~\sigma_G$.
\end{ex}

\section{Torus action on matrix Schubert varieties in terms of graphs}\label{torusactiononmsv}

We are interested in the torus action on $Y_{\pi}$. This question has been first studied by Escobar and M\'{e}sz\'{a}ros in \cite{escobarmeszaros} where they characterize all toric varieties $Y_{\pi}$. We reformulate this classification in terms of graphs. Moreover we approach the question of determining the dimension of the torus acting on $Y_{\pi}$ from a perspective of $T$-varieties. These are normal varieties with effective torus action having not necessarily a dense torus orbit. They can be considered as the generalization of toric varieties with respect to the dimension of their torus action. For more details about $T$-varieties, we refer to \cite{pdivisor}, \cite{geometryofT}.
\begin{defn}
An affine normal variety $X$ is called a \textit{T-variety} of complexity-$d$ if it admits an effective $T$ torus action with $\text{dim}(X) - \text{dim}(T) = d$.
\end{defn}

\noindent The matrix Schubert varieties are normal varieties (see \cite{factschu}, Theorem 2.4.3.). The action of $B_{\_} \times B_{+}$ on $\overline{X_{\pi}}$ restricts to the action of $T^N \times T^N$, where $T^N \cong (\CC^*)^N$ is a diagonal matrix of size $N\times N$. Since this action is not effective ($(a I_N , a I_N ) . \mathcal{M}=  \mathcal{M}$), we consider the stabilizer $\Stab((\CC^*)^{2N})$ of this torus action and the action of the quotient $ T:=(\mathbb{C}^{*})^{2N}/ \Stab((\CC^*)^{2N})$ on the matrix Schubert variety $\overline{X_{\pi}}$.\\

\noindent Let $p$ be a general point in $Y_{\pi}$ which have $1$ in all boxes of $L(\pi)$ and $0$ in others. Then the closure of the torus orbit $\overline{(\CC^*)^{2N} . p}$ is the affine toric variety associated to the so-called $(\CC^*)^{2N}$-moment cone (or weight cone) of $Y_{\pi}$, denoted by $\Phi(Y_{\pi})$. One obtains that $\dim(\Phi(Y_{\pi})) = \dim(\overline{(\CC^*)^{2N} . p})$. Since $\overline{(\CC^*)^{2N} . p}$ and $Y_{\pi}$ are both irreducible, it suffices to examine their dimension for the complexity of the torus action on $Y_{\pi}$. Recall that the convex polyhedral cone generated by all weights of the torus action on $Y_{\pi}$ in $M_{\RR}$ (vector space spanned by the character lattice of considered torus) is called the weight cone. Here, the weight cone of the action can be expressed as $$ \Phi(Y_{\pi}) = \cone(e_i - f_j \ | \ (i,j)\in L(w))$$ where $e_i$ denotes the canonical basis for $\mathbb{R}^m \times {0}$ and $f_j$ denotes the canonical basis for ${0} \times \mathbb{R}^n$.  Note that this cone is GL-equivalent to a dual edge cone associated to a bipartite graph. Hence one can define a bipartite graph $G^{\pi} \subseteq K_{m,n}$ from a Rothe diagram $D(\pi)$ via the following bijection:

\begin{center}
$L(\pi) \longrightarrow E(G^{\pi})$ \\
$(a,b) \mapsto (a,b) $ 

\end{center}

\noindent where for $(a,b) \in E(G^{\pi})$, $a \in U_1$ and $b \in U_2$. Hence we obtain also the vertex set $V(G^{\pi})$. We denote the associated edge cone by $\sigma_{\pi}$. By Remark \ref{dimensionYpi}, we conclude the following.
\begin{prop}\label{complexity}
$Y_{\pi}$ is a T-variety of complexity $d$ with respect to the torus action $T$ if and only if $ \dim( \sigma^{\vee}_{\pi})= L'(\pi) -d$. 
\end{prop}

\begin{ex}
Let us consider the matrix Schubert variety $\overline{ X_{\pi}} \cong Y_{\pi} \times \CC^7$ for $\pi=[2143]$ in Example~\ref{firstexample}. The second figure represents $L(\pi)$ and the third figure represents the bipartite graph $\G^{\pi}$. For each box $(a,b) \in L(\pi)$, we construct an edge $(a,b) \in E(G^{\pi})$ with vertices $a \in U_1$ and $b \in U_2$. The dimension of the associated dual edge cone $\sigma^{\vee}_{\pi}$ is 5 and $|L'(\pi)| = 7$. Hence $Y_{\pi}$ is a $T$-variety of complexity 2 with respect to the effective torus action of $T \cong (\CC^*)^5$ and with a moment cone linearly equivalent to $\sigma_{\pi}^{\vee}$.
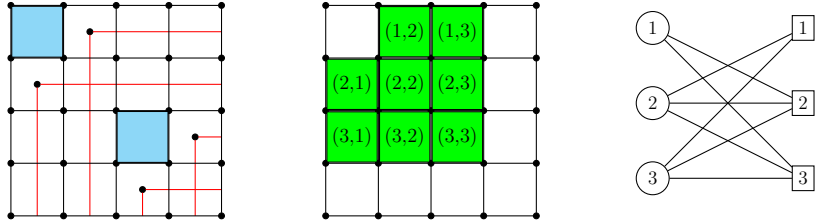
\vspace{0.5cm}
\begin{center}
\begin{tikzpicture}[baseline=1 cm,scale=0.625,every path/.style={>=latex},every node/.style={draw,circle,fill=black,scale=0.2}]
 \path[draw, fill=cyan!40,opacity=1] (0,3)--(1,3)--(1,4)--(0,4)--cycle;
 \path[draw, fill=cyan!40,opacity=1] (2,1)--(3,1)--(3,2)--(2,2)--cycle;

  \node            (a0) at (0,0)  {};
  \node            (b0) at (1,0)  {};
  \node            (c0) at (2,0) {};
  \node            (d0) at (3,0) {};
  \node            (e0) at (4,0) {};

  \node            (a1) at (0,1)  {};
  \node            (b1) at (1,1)  {};
  \node            (c1) at (2,1) {};
  \node            (d1) at (3,1) {};
  \node            (e1) at (4,1) {};

  \node            (a2) at (0,2)  {};
  \node            (b2) at (1,2)  {};
  \node            (c2) at (2,2) {};
  \node            (d2) at (3,2) {};
  \node            (e2) at (4,2) {};

  \node            (a3) at (0,3)  {};
  \node            (b3) at (1,3)  {};
  \node            (c3) at (2,3) {};
  \node            (d3) at (3,3) {};
  \node            (e3) at (4,3) {};

  \node            (a4) at (0,4)  {};
  \node            (b4) at (1,4)  {};
  \node            (c4) at (2,4) {};
  \node            (d4) at (3,4) {};
  \node            (e4) at (4,4) {};

  \node            (p1) at (0.5,2.5) {};
  \node            (p2) at (1.5,3.5) {};
  \node            (p3) at (2.5,0.5) {};
  \node            (p4) at (3.5,1.5) {};

\draw[-,red] (p1) edge (0.5,0);
\draw[-,red] (p1) edge (4,2.5);

\draw[-,red] (p2) edge (1.5,0);
\draw[-,red] (p2) edge (4,3.5);

\draw[-,red] (p3) edge (2.5,0);
\draw[-,red] (p3) edge (4,0.5);

\draw[-,red] (p4) edge (3.5,0);
\draw[-,red] (p4) edge (4,1.5);

 \draw[-] (a0) edge (e0);
  \draw[-] (a1) edge (e1);
   \draw[-] (a2) edge (e2);
    \draw[-] (a3) edge (e3);
    \draw[-] (a4) edge (e4);

 \draw[-] (a0) edge (a4); 
 \draw[-] (b0) edge (b4); 
 \draw[-] (c0) edge (c4); 
 \draw[-] (d0) edge (d4); 
 \draw[-] (e0) edge (e4); 

\end{tikzpicture}
\hspace{1cm}
\begin{tikzpicture}[baseline=1cm,scale=0.625,every path/.style={>=latex},every node/.style={draw,circle,fill=black,scale=0.2}]
 \path[draw, fill=LimeGreen!40,opacity=1] (0,1)--(2,1)--(3,1)--(3,2)--(3,4)--(1,4)--(1,3)--(0,3)--cycle;
  \node            (a0) at (0,0)  {};
  \node            (b0) at (1,0)  {};
  \node            (c0) at (2,0) {};
  \node            (d0) at (3,0) {};
  \node            (e0) at (4,0) {};

  \node            (a1) at (0,1)  {};
  \node            (b1) at (1,1)  {};
  \node            (c1) at (2,1) {};
  \node            (d1) at (3,1) {};
  \node            (e1) at (4,1) {};

  \node            (a2) at (0,2)  {};
  \node            (b2) at (1,2)  {};
  \node            (c2) at (2,2) {};
  \node            (d2) at (3,2) {};
  \node            (e2) at (4,2) {};

  \node            (a3) at (0,3)  {};
  \node            (b3) at (1,3)  {};
  \node            (c3) at (2,3) {};
  \node            (d3) at (3,3) {};
  \node            (e3) at (4,3) {};

  \node            (a4) at (0,4)  {};
  \node            (b4) at (1,4)  {};
  \node            (c4) at (2,4) {};
  \node            (d4) at (3,4) {};
  \node            (e4) at (4,4) {};

  \node  [fill=none,draw=none,scale=3.18]     (Ca2) at (2.5,1.5) {(3,3)};

  \node  [fill=none,draw=none,scale=3.15]     (C3a) at (0.5,2.5) {(2,1)};

  \node  [fill=none,draw=none,,scale=3.15]     (C4a) at (1.5,3.5) {(1,2)};

  \node  [fill=none,draw=none,scale=3.15]     (Ca5) at (1.5,2.5) {(2,2)};

  \node  [fill=none,draw=none,scale=3.15]     (Ca6) at (0.5,1.5) {(3,1)};

  \node  [fill=none,draw=none,scale=3.15]     (Ca7) at (1.5,1.5) {(3,2)};

  \node  [fill=none,draw=none,scale=3.18]     (Ca9) at (2.5,2.5) {(2,3)};

  \node  [fill=none,draw=none,scale=3.18]     (Ca10) at (2.5,3.5) {(1,3)};

 \draw[-] (a0) edge (e0);
  \draw[-] (a1) edge (e1);
   \draw[-] (a2) edge (e2);
    \draw[-] (a3) edge (e3);
    \draw[-] (a4) edge (e4);

 \draw[-] (a0) edge (a4); 
 \draw[-] (b0) edge (b4); 
 \draw[-] (c0) edge (c4); 
 \draw[-] (d0) edge (d4); 
 \draw[-] (e0) edge (e4); 

\end{tikzpicture}
\hspace{1cm}
\begin{tikzpicture}[baseline=1.3cm,scale=1,every path/.style={>=latex},every node/.style={draw,circle,fill=white,scale=0.6}]
  \node            (a2) at (0,1.5)  {2};
  \node            (a3) at (0,0.5)  {3};

\node       [rectangle]     (b7) at (2,0.5)  {3}; 
 \node       [rectangle]     (b6) at (2,1.5)  {2};
  \node            (a1) at (0,2.5) {1};
  \node       [rectangle]     (b5) at (2,2.5) {1};

    \draw[-] (a1) edge (b7);
  \draw[-] (a1) edge (b6);
  \draw[-] (a2) edge (b5);
  \draw[-] (a2) edge (b6);
    \draw[-] (a2) edge (b7);
    \draw[-] (a3) edge (b7);
    \draw[-] (a3) edge (b6);
    \draw[-] (a3) edge (b5);

\end{tikzpicture}
\end{center}

\end{ex}
\vspace{0.5cm}

\noindent For the complexity zero case, i.e.\ toric case, we present an alternative proof with edge cones.

\begin{thm}[{\cite[Theorem 3.4]{escobarmeszaros}}]\label{toricschu}
$Y_{\pi}$ is a toric variety if and only if $L'(\pi)$ consists of disjoint hooks not sharing a row or a column.
\end{thm}

\begin{proof}
By Proposition \ref{complexity}, we aim to characterize the case when $\dim(\sigma^{\vee}_{\pi}) = L'(\pi)$. Note that $L(\pi)$ has a skew shape. Assume that $L(\pi)$ consists of $k$ connected components with $m_i$ rows and $n_i$ columns for each $i \in[k]$. This means that we investigate the bipartite graph $G^{\pi}\subseteq K_{m,n}$ with $k$ connected bipartite graph components $G^{\pi}_i \subseteq K_{m_i,n_i}$. By Proposition \ref{dimensionprop}, the dimension of the cone $\dim(\sigma_{\pi}) $ is $m+n-k$. Since $L(\pi)$ has $k$ connected components, the components of $L'(\pi)$ for each $i \in [k]$ do not share a row or a column.    Therefore, we are left with proving the statement for a connected component $L_i (\pi)$ of $L(\pi)$. The dimension of the dual edge cone of $G^{\pi}_i$ is equal to $|L'_i(\pi)|$ if and only if $L'_i(\pi)$ has a hook shape. 
\end{proof}

\begin{ex}
Let $\pi = [2413] \in S_4$. The first figure illustrates the Rothe diagram $D(\pi)$. The green colored boxes are $L(\pi)$ and the purple colored boxes are $L'(\pi)$. The dimension of the associated bipartite graph and $|L'(\pi)|$ is three. Also, as seen in the last figure, $L'(\pi)$ has a hook shape. Thus, $Y_{[2413]}$ is a toric variety with respect to the effective torus action of $T \cong (\CC^*)^3$, in particular the cone over the Segre variety $\mathbb P^1 \times \mathbb P^1$.
\vspace{0.2cm}
\begin{center}
\begin{tikzpicture}[baseline=1cm, scale=0.625,every path/.style={>=latex},every node/.style={draw,circle,fill=black,scale=0.2}]
 \path[draw, fill=cyan!40,opacity=1] (0,3)--(2,3)--(2,4)--(0,4)--cycle;
 \path[draw, fill=cyan!40,opacity=1] (1,1)--(2,1)--(2,2)--(1,2)--cycle;
  \node            (a0) at (0,0)  {};
  \node            (b0) at (1,0)  {};
  \node            (c0) at (2,0) {};
  \node            (d0) at (3,0) {};
  \node            (e0) at (4,0) {};

  \node            (a1) at (0,1)  {};
  \node            (b1) at (1,1)  {};
  \node            (c1) at (2,1) {};
  \node            (d1) at (3,1) {};
  \node            (e1) at (4,1) {};

  \node            (a2) at (0,2)  {};
  \node            (b2) at (1,2)  {};
  \node            (c2) at (2,2) {};
  \node            (d2) at (3,2) {};
  \node            (e2) at (4,2) {};

  \node            (a3) at (0,3)  {};
  \node            (b3) at (1,3)  {};
  \node            (c3) at (2,3) {};
  \node            (d3) at (3,3) {};
  \node            (e3) at (4,3) {};

  \node            (a4) at (0,4)  {};
  \node            (b4) at (1,4)  {};
  \node            (c4) at (2,4) {};
  \node            (d4) at (3,4) {};
  \node            (e4) at (4,4) {};

  \node            (p1) at (0.5,2.5) {};
  \node            (p2) at (2.5,3.5) {};
  \node            (p3) at (1.5,0.5) {};
  \node            (p4) at (3.5,1.5) {};

\draw[-,red] (p1) edge (0.5,0);
\draw[-,red] (p1) edge (4,2.5);

\draw[-,red] (p2) edge (2.5,0);
\draw[-,red] (p2) edge (4,3.5);

\draw[-,red] (p3) edge (1.5,0);
\draw[-,red] (p3) edge (4,0.5);

\draw[-,red] (p4) edge (3.5,0);
\draw[-,red] (p4) edge (4,1.5);

 \draw[-] (a0) edge (e0);
  \draw[-] (a1) edge (e1);
   \draw[-] (a2) edge (e2);
    \draw[-] (a3) edge (e3);
    \draw[-] (a4) edge (e4);

 \draw[-] (a0) edge (a4); 
 \draw[-] (b0) edge (b4); 
 \draw[-] (c0) edge (c4); 
 \draw[-] (d0) edge (d4); 
 \draw[-] (e0) edge (e4); 

\end{tikzpicture}
\hspace{1cm}
\begin{tikzpicture}[baseline=1cm,scale=0.625,every path/.style={>=latex},every node/.style={draw,circle,fill=black,scale=0.2}]
 \path[draw, fill=LimeGreen!40,opacity=1] (0,1)--(2,1)--(2,3)--(0,3)--cycle;
  \node            (a0) at (0,0)  {};
  \node            (b0) at (1,0)  {};
  \node            (c0) at (2,0) {};
  \node            (d0) at (3,0) {};
  \node            (e0) at (4,0) {};

  \node            (a1) at (0,1)  {};
  \node            (b1) at (1,1)  {};
  \node            (c1) at (2,1) {};
  \node            (d1) at (3,1) {};
  \node            (e1) at (4,1) {};

  \node            (a2) at (0,2)  {};
  \node            (b2) at (1,2)  {};
  \node            (c2) at (2,2) {};
  \node            (d2) at (3,2) {};
  \node            (e2) at (4,2) {};

  \node            (a3) at (0,3)  {};
  \node            (b3) at (1,3)  {};
  \node            (c3) at (2,3) {};
  \node            (d3) at (3,3) {};
  \node            (e3) at (4,3) {};

  \node            (a4) at (0,4)  {};
  \node            (b4) at (1,4)  {};
  \node            (c4) at (2,4) {};
  \node            (d4) at (3,4) {};
  \node            (e4) at (4,4) {};

 \draw[-] (a0) edge (e0);
  \draw[-] (a1) edge (e1);
   \draw[-] (a2) edge (e2);
    \draw[-] (a3) edge (e3);
    \draw[-] (a4) edge (e4);

 \draw[-] (a0) edge (a4); 
 \draw[-] (b0) edge (b4); 
 \draw[-] (c0) edge (c4); 
 \draw[-] (d0) edge (d4); 
 \draw[-] (e0) edge (e4); 

\end{tikzpicture}
\hspace{1cm}
\begin{tikzpicture}[baseline=1.3cm,scale=1.3,every path/.style={>=latex},every node/.style={draw,circle,fill=white,scale=0.7}]
  \node            (a) at (0,0.5)  {2};
  \node       [rectangle]     (b) at (1.5,0.5)  {2};
  \node            (c) at (0,2) {1};
  \node         [rectangle]   (d) at (1.5,2) {1};

    \draw[-] (c) edge (b);
    \draw[-] (b) edge (a);
    \draw[-] (a) edge (d);
  \draw[-] (c) edge (d);
\end{tikzpicture}
\hspace{1cm}
\begin{tikzpicture}[baseline=1cm,scale=0.625,every path/.style={>=latex},every node/.style={draw,circle,fill=black,scale=0.2}]
 \path[draw, fill=Orchid!40,opacity=1] (0,1)--(1,1)--(1,2)--(2,2)--(2,3)--(0,3)--cycle;
  \node            (a0) at (0,0)  {};
  \node            (b0) at (1,0)  {};
  \node            (c0) at (2,0) {};
  \node            (d0) at (3,0) {};
  \node            (e0) at (4,0) {};

  \node            (a1) at (0,1)  {};
  \node            (b1) at (1,1)  {};
  \node            (c1) at (2,1) {};
  \node            (d1) at (3,1) {};
  \node            (e1) at (4,1) {};

  \node            (a2) at (0,2)  {};
  \node            (b2) at (1,2)  {};
  \node            (c2) at (2,2) {};
  \node            (d2) at (3,2) {};
  \node            (e2) at (4,2) {};

  \node            (a3) at (0,3)  {};
  \node            (b3) at (1,3)  {};
  \node            (c3) at (2,3) {};
  \node            (d3) at (3,3) {};
  \node            (e3) at (4,3) {};

  \node            (a4) at (0,4)  {};
  \node            (b4) at (1,4)  {};
  \node            (c4) at (2,4) {};
  \node            (d4) at (3,4) {};
  \node            (e4) at (4,4) {};

 \draw[-] (a0) edge (e0);
  \draw[-] (a1) edge (e1);
   \draw[-] (a2) edge (e2);
    \draw[-] (a3) edge (e3);
    \draw[-] (a4) edge (e4);

 \draw[-] (a0) edge (a4); 
 \draw[-] (b0) edge (b4); 
 \draw[-] (c0) edge (c4); 
 \draw[-] (d0) edge (d4); 
 \draw[-] (e0) edge (e4); 

\end{tikzpicture}

\end{center}
\end{ex}

\vspace{0.5cm}

\noindent Theorem \ref{toricschu} and its alternative proof give us the opportunity to study the first-order deformations of $Y_{\pi}$ in terms of edge cone $\sigma_{\pi}$ and Rothe diagram $D(\pi)$ in Section \ref{section4}.

\section{Rigidity of toric matrix Schubert varieties}\label{section4}

This section is devoted to the study of the detailed structure of the edge cone $\sigma_{\pi}$ for matrix Schubert varieties $\overline{X_{\pi}}$ where $Y_{\pi}=\TV(G^{\pi})$ is toric. Note that these matrix Schubert varieties are called toric matrix Schubert varieties in \cite{escobarmeszaros} and we adopt this convention. First, we explain briefly the combinatorial techniques for the first order deformations of toric varieties. By studying the first independent sets of $G^{\pi}$ and the two and three-dimensional faces of $\sigma_{\pi}$, we present the conditions for rigidity of toric matrix Schubert varieties. By Remark \ref{notconnected} and since we investigate rigidity, we can assume that $L(\pi)$ is connected. Throughout this section, the connected bipartite graph $G^{\pi} \subseteq K_{m,n}$ denotes the associated bipartite graph of $L(\pi)$ which was constructed in Section \ref{torusactiononmsv}.

\subsection{Deformations of toric varieties}\label{deformationtheory}

A deformation of an affine algebraic variety $X_0$ is a flat morphism $\pi \colon \mathcal{X} \longrightarrow S$ with $0 \in S$ such that $\pi^{-1}(0) = X_0$, i.e.\  we have the following
 commutative diagram.
\begin{center}
\begin{tikzpicture}[every node/.style={midway}]
\matrix[column sep={4em,between origins},
        row sep={2em}] at (0,0)
{ \node(X0)   {$X_0$}  ; & \node(X) {$\mathcal{X}$}; \\
  \node(0) {$0$};&      \node(S) {$S$};             \\};
\draw[->] (X0) -- (0) node[anchor=east]  {$$};
\draw[->] (X) -- (S) node[anchor=west]  {$\pi$};
\draw[right hook->] (X0) -- (X) node[anchor=north]  {$$};
\draw[right hook->] (0)   -- (S) node[anchor=south] {$$};
\end{tikzpicture}
\end{center}

\noindent The variety $\mathcal{X}$ is called the total space and $S$ is called the base space of the deformation. Let $\pi \colon \mathcal{X} \longrightarrow S $ and $\pi' \colon \mathcal{X}' \longrightarrow S$ be two deformations of $X_0$. We say that two deformations are isomorphic if there exists a map $\phi: \mathcal{X} \longrightarrow \mathcal{X}'$ over $S$ inducing the identity on $X_0$. Let $A$ be an Artin ring and let $S = \Spec(A)$. One has a contravariant functor $\Def_{X_0}$ such that $\Def_{X_0} (A)$ is the set of deformations of $X_0$ over $S$ modulo isomorphisms. 
\begin{defn}
The map $\pi$ is called a first order deformation of $X_0$ if $S=\Spec(\CC [\epsilon]/(\epsilon^2))$. We set {$T^1_{X_0}$}$ := \Def_{X_0}(\CC [\epsilon] /(\epsilon^2))$.
\end{defn}

\noindent The variety $X_0$ is called {\it rigid} if $T_{X_0}^{1} = 0$. This implies that a rigid variety $X_0$ has no nontrivial infinitesimal deformations. This means that every deformation $\pi \in \Def_{X_0}(A)$ over $S=\Spec(A)$ is trivial i.e.\ isomorphic to the trivial deformation $X_0 \times S \longrightarrow S$. \\

\noindent For the case where $X_0 = \Spec(\CC[\sigma^{\vee} \cap M])$ is an affine normal toric variety, we introduce the techniques which are developed in \cite{altmann96} in order to investigate the $\CC$-vector space $T_{X_0}^1$. The deformation space $T_{X_0}^1$ is multigraded by the lattice elements of $M$, i.e.\ $T_{X_0}^1 = \bigoplus_{R\in M} T^1_{X_0}(-R)$. We first set some definitions in order to describe the homogeneous part $T^1_{X_0}(-R)$.

\begin{defn}
Let us call $R\in M$ a deformation degree and let $\sigma\subseteq N_{\QQ}$ be generated by the extremal ray generators $a_1,\ldots,a_k$. We define the following affine space \[[R=1]:=\{a \in N_\mathbb{Q} \ | \  \langle R, a \rangle =1\} \subseteq N_{\mathbb{Q}}.\]
The cross-cut of $\sigma$ in degree $R$ is the polyhedron {$Q(R)$}$:=\sigma \cap [R=1]$ in the assigned vector space $[R=0]:=\{a \in N_\mathbb{Q} \ | \  \langle R, a \rangle =0\} \subseteq N_{\mathbb{Q}}$.
\end{defn} 
The cross-cut $Q(R)$ has the cone of unbounded directions $Q(R)^{\infty} = \sigma \cap [R=0]$ and the minimal convex compact part $Q(R)^{c}$ such that  $Q(R)=Q(R)^{c}+Q(R)^{\infty}$ holds. The compact part $Q(R)^{c}$ is generated by the vertices $\overline{a_i} = a_i /\langle R, a_i \rangle$ where $\langle R, a_i \rangle \geq 1$. Note that $\overline{a_i}$ is a lattice vertex in $Q(R)$ if $\langle R, a_i \rangle =1$.

\begin{defn}

\begin{itemize}
\item[(i)] Let $d^1,\hdots,d^{\alpha} \in R^{\bot} \subset N_{\QQ}$ be the compact edges of $Q(R)$. The vector $\bar{\epsilon} \in \{0, \pm 1\}^{\alpha}$ is called a sign vector assigned to each two-dimensional compact face $\epsilon$ of $Q(R)$ defined as  
\[
    \overline{\epsilon_i}=\left\{
                \begin{array}{ll}
                  \pm 1, \ \text{if} \ d^i \ \text{is an edge of} \ \epsilon\\
                  0\\
                \end{array}
              \right.  \]

\noindent such that $\sum_{i \in [\alpha]} \overline{\epsilon_i} d^i=0$, i.e the oriented edges $\overline{\epsilon_i} d^i$ form a cycle along the edges of $\epsilon$. We choose one of both possibilities for the sign of $\overline{\epsilon}$.\\
\item[(ii)] For every deformation degree $R \in M$, the related vector space is defined as $$V(R)=\{\overline{t}=(t_1,\hdots,t_{\alpha}) \in \CC^{\alpha} :   \sum_{i \in [\alpha]} t_i \overline{\epsilon_i} d^i = 0, \text{ for every compact 2-face }\epsilon \preceq Q(R)\}.$$
\end{itemize}
\end{defn}
\noindent The toric variety $\TV(G)$ associated to a bipartite graph $G \subseteq K_{m,n}$ is smooth in codimension 2 (\cite{rigidportakal}, Theorem 4.5). Hence we introduce the formula for this special case.
\begin{thm}[{\cite[Corollary 2.7]{altmann96}}]\label{altmann96}
If the affine normal toric variety $X_0$ is smooth in codimension 2, then $T^1_{X_0}(-R)$ is contained in $V(R)/\mathbb{C}(1, \ldots, 1)$. Moreover, it is built by those $\bar{t}$'s  satisfying $t_{ij} = t_{jk}$ where $\overline{a_j}$ is a non-lattice common vertex in $Q(R)$ of the edges $d^{ij}=\overline{\overline{a_i} \ \overline{a_j}}$ and $d^{jk}=\overline{\overline{a_j} \ \overline{a_k}}$. 
\end{thm}

\begin{rem} 
The following two cases of $Q(R)$ in Figure \ref{transferringexplained} will appear often while we study $T^1_{\TV(G)}(-R)$. Let us interpret these cases with the previous result. 
 
\begin{center}
\begin{tikzpicture}[baseline=1cm,scale=0.7,every path/.style={>=latex},every node/.style={draw,circle,fill=black,scale=0.2}]
  \node      [draw=black, circle, fill=black, scale=3]      (a) at (0,0)  {};
  \node        [draw=black, circle, fill=black, scale=3]      (b) at (-3,2)  {};
  \node       [draw=black, circle, fill=black, scale=3]        (d) at (3,2) {};
  \node        [draw=black, circle, fill=black, scale=3]       (f) at (0,4) {};

  \node      [draw=none, fill=white, scale=5]      (a1) at (-0.8,1.1)  {$d^1$};
  \node      [draw=none, fill=white, scale=5]      (a4) at (1,1.1)  {$d^5$};
  \node      [draw=none, fill=white, scale=5]      (a7) at (0.3,2)  {$d^3$};
  \node      [draw=none, fill=white, scale=5]      (a8) at (-1,3)  {$d^2$};
  \node      [draw=none, fill=white, scale=5]      (a9) at (1.2,2.7)  {$d^4$};

    \draw[->,green] (a) edge (b);

 \draw [<-,green] (f) edge (b);
 \draw [<-,green] (f) edge (d);
 \draw [->,green] (f) edge (a);

\draw [->,green] (a) edge (d);

\end{tikzpicture}
\hspace{2cm}
\begin{tikzpicture}[baseline=1,scale=0.7,every path/.style={>=latex},every node/.style={draw,circle,fill=black,scale=0.2}]
  \node      [draw=black, circle, fill=red, scale=3]      (a) at (0,0)  {$\overline{a_j}$};
  \node        [draw=black, circle, fill=black, scale=3]      (b) at (-2,2)  {};
  \node         [draw=black, circle, fill=black, scale=3]       (c) at (-3,-1.5) {};
  \node       [draw=black, circle, fill=black, scale=3]        (d) at (2,2) {};
  \node        [draw=black, circle, fill=black, scale=3]      (e) at (3,-1.5) {};

  \node      [draw=none, fill=white, scale=5]       (X) at (-1.5,1)  {$d^1$};
  \node        [draw=none, fill=white, scale=5]       (Y) at (-2.2,0.2)  {$d^2$};
  \node         [draw=none, fill=white, scale=5]      (Z) at (-1.5,-0.2) {$d^3$};
  \node      [draw=none, fill=white, scale=5]       (X1) at (1,0.5) {$d^4$};
  \node       [draw=none, fill=white, scale=5]       (Y1) at (2,0.8) {$d^5$};
  \node        [draw=none, fill=white, scale=5]     (Z1) at (1.7,-0.2) {$d^6$};

    \draw[->,green] (a) edge (b);

\draw [->,green] (b) edge (c);

\draw [<-,green] (a) edge (c);
\draw [->,green] (a) edge (d);

\draw [<-,green] (a) edge (e);

\draw [->,green] (d)  edge (e);

\end{tikzpicture}

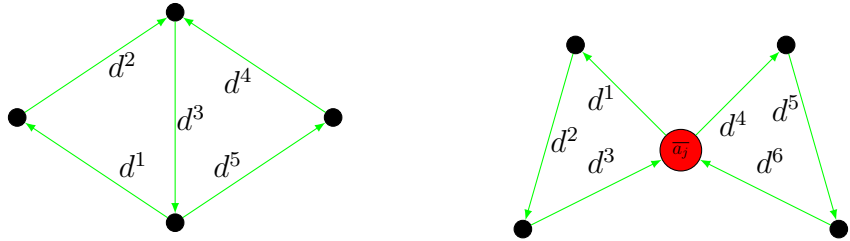
\captionof{figure}{Compact 2-faces sharing an edge or a non-lattice vertex in $Q(R)$}
\label{transferringexplained}
\end{center}

\vspace{0.5cm}

\noindent $\bullet$ Let $\epsilon^1, \epsilon^2 \preceq Q(R)$ be the compact 2-faces sharing the edge $d^3$. We choose the sign vectors $\overline{\epsilon^1} = (1,1,1,0,0)$ and $\overline{\epsilon^2} = (0,0,1,1,1)$. Suppose that $\overline{t} = (t_1, t_2, t_3, t_4, t_5) \in V(R)$. We observe that $t_1 = t_2 = t_3$ for 2-face $\epsilon^1$ and $t_3 = t_4 = t_5$ for 2-face $\epsilon^2$.\\
\noindent $\bullet$ Let $\epsilon^1, \epsilon^2 \preceq Q(R)$ be the compact 2-faces connected by the vertex $\overline{a_j}$. As in the previous case we obtain that $t_1 = t_2 = t_3$ and $t_4 = t_5 = t_6$. By Theorem \ref{altmann96}, if $a_j$ is a non-lattice vertex, then we obtain $t_3 = t_4$. 
\end{rem}

\subsection{Faces of the edge cone $\sigma_{\pi}$ of toric variety $Y_{\pi}$}\label{rigiditysection}

In order to study the rigidity of $Y_{\pi}=~\TV(\sigma_{\pi})$ with Theorem \ref{altmann96}, we investigate the face structure of the edge cone $\sigma_{\pi}$ more closely. We consider three types of first independent sets with following notations: the one-sided first independent sets $A = U_1\backslash \{\bullet\}$, $B= U_2 \backslash \{\bullet\}$ and two-sided (maximal) first independent sets $C = C_1 \sqcup C_2$. We label the essential boxes from the bottom of the diagram starting with $(x_1,y_1)$ to the top ending with $(x_{k+1}, y_{k+1})$ i.e.\ we have $x_{k+1} < \ldots < x_1$ and $y_1 < \ldots < y_{k+1}$. 

\begin{lem}\label{schulemma1}
For any permutation $\pi \in S_N$,
\begin{enumerate}
\item[\normalfont (1)] The one-sided first independent sets of $G^\pi$ are $U_i \backslash \{u_i\}$ for all $u_i \in U_i$ and for $i =1,2$. 
\item[\normalfont (2)] The two-sided first independent sets are the maximal two-sided independent sets of $G^{\pi}$.
\end{enumerate}
\end{lem}

\begin{proof}
By Theorem \ref{toricschu}, $L'(\pi)$ is a hook. The boxes of $L(\pi)$ form a shape of a Ferrer diagram, i.e.\ we have $\lambda_1 \geq \ldots \geq \lambda_t$ where $\lambda_i$ denotes the number of boxes at $i$th row of $L(\pi)$.  Consider the smallest rectangle containing $L(\pi)$ of length $m$ and of width $n$. The removed edges of the bipartite graph $G^{\pi}\subseteq K_{m,n}$ are linked with the free boxes of $\overline{X_{\pi}}$ in the rectangle. Let $(x_i,y_i) \in \Ess(\pi)$, equivalently let $(x_i,y_i) \in E(G^{\pi})$.  Then one obtains naturally that there exists a two-sided maximal independent set $C = C_1 \sqcup C_2 = \{x_{i}+1, \ldots,m\} \sqcup \{y_{i-1}+1, \ldots, n\}$ where $(x_{i-1},y_{i-1}) \in \Ess(\pi)$ with $x_{i-1} > x_i$ and $y_{i-1} <y_{i}$. Then the neighbor sets are $N(C_1)=U_2 \backslash C_2 =\{1, \ldots, y_{i-1}\}$ and $N(C_2) = U_1 \backslash C_1 = \{1, \ldots , x_{i}\}$. Therefore, the boxes for the induced subgraphs $\G[C_1 \sqcup N(C_1)]$ and $\G[C_2 \sqcup N(C_2)]$ also form a shape of a Ferrer diagram and $\G\{C\}$ has two connected components.  In particular, $U_i \backslash \{u_i\}$ cannot be contained in a two-sided independent set. Suppose that $\G\{U_i \backslash \{u_i\}\}$ has more than three components. Then as in \cite[Proposition 2.13]{rigidportakal}, there exist two-sided first independent sets $C^{i} \in \Ione$ such that $\bigsqcup C_1^{i} = U_i \backslash \{u_i\}$  which is not possible. 
\end{proof}

\begin{lem}\label{mainlemschu}
There exist $k$ two-sided first independent sets of $G^{\pi}$ with $|\Ess(\pi) | = k+1$. Moreover, if $k \geq 2$ and, $C$ and $C'$ are two-sided first independent sets of $G^{\pi}$, then $C_1 \subsetneq C'_1$ and $C'_2 \subsetneq C_2$. 
\end{lem}

\begin{proof}
Consider again the smallest rectangle containing $L(\pi)$ of a length $m$ and of a width $n$. If there exists only one essential set of $\pi$, then $G^{\pi} = K_{m,n}$. Assume that there are more than one essential box. Let $(x_{j}, y_{j})$ and $(x_i, y_i)$ be two essential boxes with $j<i$, $x_{j}>x_i$ and $y_{j}<y_i$. By Lemma \ref{schulemma1}, we obtain two first independent sets $C = \{x_{i}+1, \ldots,m\} \sqcup \{y_{i-1} +1, \ldots, n\}$ and $C' = \{x_{j}+1, \ldots,m\} \sqcup \{y_{j-1}+1, \ldots, n\}$ of $G^{\pi}$. We infer that $C_1 \subsetneq C'_1$ and $C'_2 \subsetneq C_2$. 
\end{proof}

\begin{ex}\label{rigidexample}

The boxes of $L(\pi)$ for the toric variety $Y_{\pi}$ is presented in Figure \ref{youngtableauxfig}. The blue boxes are removed edges between some vertex sets $C_1 \subset U_1$ and $C_2 \subset U_2$. We observe  that $C:=~C_1 \sqcup C_2$ is a maximal independent set. In particular, the orange color represents the edges of the induced subgraph $\G[C_1 \sqcup N(C_1)]$ and the purple color represents the edges of the induced subgraph $\G[C_2 \sqcup N(C_2)]$. The crossed boxes are the boxes of the essential set $\Ess(\pi)$. The boxes with a dot form the shape of a hook and these are the boxes of $L'(\pi)$.
\begin{center}
\ytableausetup{smalltableaux}
\begin{ytableau}
*(white) \bullet & *(white)  \bullet &*(white) \bullet  &*(white) \bullet  & *(white) \bullet  &*(white) \bullet &*(white)\bullet  & *(white)\bullet   &*(white) \bullet &*(Orchid) \bullet & *(Orchid) \bullet  &*(Orchid)\bullet  &*(Orchid)\bullet  & *(Orchid)\bullet   &*(Orchid)\bullet  & *(Orchid)\bullet   &*(Orchid) \bullet  \\
*(white)  \bullet & *(white)  &*(white) &*(white) & *(white)  &*(white) &*(white) & *(white)  &*(white) &*(Orchid) &*(Orchid)  &*(Orchid)&*(Orchid)  &*(Orchid)  &*(Orchid)&*(Orchid)  &*(Orchid)   \\
*(white) \bullet  & *(white)  &*(white) &*(white) & *(white)  &*(white) &*(white) & *(white)  &*(white) &*(Orchid) &*(Orchid)  &*(Orchid)&*(Orchid) &*(Orchid)&*(Orchid)&*(Orchid)  &*(Orchid) \times  \\
*(white) \bullet & *(white)  &*(white) &*(white) & *(white)  &*(white) &*(white) & *(white)  &*(white) &*(Orchid) &*(Orchid)  &*(Orchid)&*(Orchid)  &*(Orchid)   \\
*(white) \bullet & *(white)  &*(white) &*(white) & *(white)  &*(white) &*(white) & *(white)  &*(white) &*(Orchid) &*(Orchid)  &*(Orchid)&*(Orchid)  &*(Orchid) \times \\
*(white)  \bullet & *(white)  &*(white) &*(white) & *(white)  &*(white) &*(white) & *(white)  &*(white) & *(Orchid)  &*(Orchid)  \\
*(white) \bullet & *(white)  &*(white) &*(white) & *(white)  &*(white) &*(white) & *(white)  &*(white)  & *(Orchid)  &*(Orchid)\times \\
*(BurntOrange)  \bullet & *(BurntOrange) &*(BurntOrange) &*(BurntOrange) & *(BurntOrange) &*(BurntOrange) &*(BurntOrange) & *(BurntOrange)  &*(BurntOrange)& *(cyan!40)  &*(cyan!40) &*(cyan!40) & *(cyan!40) &*(cyan!40)&*(cyan!40) & *(cyan!40) &*(cyan!40)\\
*(BurntOrange)  \bullet & *(BurntOrange) &*(BurntOrange) &*(BurntOrange) & *(BurntOrange) &*(BurntOrange) &*(BurntOrange) & *(BurntOrange)  &*(BurntOrange) &  *(cyan!40)  &*(cyan!40) &*(cyan!40) & *(cyan!40) &*(cyan!40)&*(cyan!40) & *(cyan!40) &*(cyan!40)\\
*(BurntOrange)  \bullet & *(BurntOrange) &*(BurntOrange) &*(BurntOrange) & *(BurntOrange) &*(BurntOrange) &*(BurntOrange) & *(BurntOrange)  &*(BurntOrange) \times & *(cyan!40)  &*(cyan!40) &*(cyan!40) & *(cyan!40) &*(cyan!40)&*(cyan!40) & *(cyan!40) &*(cyan!40)\\
*(BurntOrange)  \bullet & *(BurntOrange) &*(BurntOrange) &*(BurntOrange) & *(BurntOrange) &*(BurntOrange) & \none  &\none &\none & *(cyan!40) &*(cyan!40)&*(cyan!40) & *(cyan!40) &*(cyan!40) & *(cyan!40)  &*(cyan!40) &*(cyan!40) \\
*(BurntOrange)  \bullet & *(BurntOrange) &*(BurntOrange) &*(BurntOrange) & *(BurntOrange) &*(BurntOrange)\times & \none  & \none &\none & *(cyan!40) &*(cyan!40)&*(cyan!40) & *(cyan!40)  &*(cyan!40) & *(cyan!40)  &*(cyan!40) &*(cyan!40)\\
*(BurntOrange)  \bullet & *(BurntOrange) &*(BurntOrange) & \none  & \none &\none & \none  & \none &\none & *(cyan!40) &*(cyan!40)&*(cyan!40) & *(cyan!40) &*(cyan!40)& *(cyan!40)  &*(cyan!40) &*(cyan!40)\\
*(BurntOrange)  \bullet & *(BurntOrange) &*(BurntOrange)  & \none  & \none &\none  & \none  & \none &\none & *(cyan!40) &*(cyan!40)&*(cyan!40) & *(cyan!40) &*(cyan!40)& *(cyan!40)  &*(cyan!40) &*(cyan!40)\\
*(BurntOrange)  \bullet & *(BurntOrange)& *(BurntOrange)& \none& \none  & \none   & \none  & \none &\none & *(cyan!40) &*(cyan!40)&*(cyan!40) & *(cyan!40) &*(cyan!40)& *(cyan!40)  &*(cyan!40) &*(cyan!40)\\
*(BurntOrange)  \bullet & *(BurntOrange) & *(BurntOrange) \times & \none & \none  &\none  & \none  & \none &\none & *(cyan!40) &*(cyan!40)&*(cyan!40) & *(cyan!40) &*(cyan!40)& *(cyan!40)  &*(cyan!40) &*(cyan!40)
\end{ytableau}
\end{center}


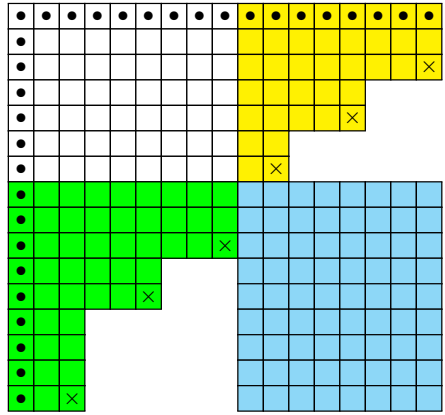
\captionof{figure}{A representative figure of a first independent set of $C \in \mathcal{I}^{(1)}_{G^\pi}$ and $\G\{C\}$ for a matrix Schubert variety $\overline{X_{\pi}}$}\label{youngtableauxfig}
\end{ex}

\noindent Let us first identify the cases where there is one or there are two essential boxes. 
\begin{lem}\label{eskirigids}
Let $G^{\pi} \subseteq K_{m,n}$ be the associated connected bipartite graph to the toric variety $Y_{\pi}$.
\begin{enumerate}
\item[\normalfont (1)] If $|\Ess(\pi)| = 1$, then the toric variety $Y_{\pi}$ is isomorphic to $\TV(K_{m.n})$, i.e.\ the cone over the Segre variety $\mathbb P^{m-1} \times \mathbb P^{n-1}$. In particular, $Y_{\pi}$ is rigid if $m \neq 2$ and $n \neq 2$. 
\item[\normalfont (2)] If $|\Ess(\pi)| =2$, then the toric variety $Y_{\pi}$ is rigid if and only if $|C_1|\neq 1$ and $|C_2| \neq n-2$ or $|C_1|\neq m-2$ and $|C_2| \neq 1$.
\end{enumerate}
\end{lem}

\begin{proof}
It follows from \cite[Theorem 4.3, 4.6]{rigidportakal}.
\end{proof}
\vspace{0.2cm}
\noindent From now on, we assume that $|\Ess(\pi)| \geq 3$. This means that we consider the associated connected bipartite graph $G^{\pi} \subsetneq K_{m,n}$ with $m,n \geq 4$. We denote by $\mathcal{I}_{G^\pi}^{(d)}$ the set of tuples of first independent sets forming a $d$-dimensional face of $\sigma_G$. Let $\sigma_{G^{\pi}}^{(d)}$ be the set of $d$-dimensional faces of $\sigma_{\pi}$. Recall the classification of $d$-dimensional faces of an edge cone in Theorem \ref{facetheorem} for a subset $S = \{A^{(1)}, \ldots, A^{(d)}\} \subseteq \mathcal{I}_G^{(1)}$ of $d$ first independent sets. Let $\Pi$ be the isomorphism from Theorem \ref{11thm}. Then we have
\begin{eqnarray*}
 \Pi^{d} \colon \mathcal{I}_{G^{\pi}}^{(d)} &\longrightarrow &\sigma_{G^{\pi}}^{(d)} \\
(A^{(1)}, \ldots, A^{(d)}) &\mapsto& (\Pi(A^{(1)}), \ldots, \Pi(A^{(d)})) = H_{\Val_S} \cap \sigma_{G^\pi}
\end{eqnarray*}
if and only if $\G[S]=\bigcap_{A \in S} \G\{A\}$ has $d+1$ connected components.
\begin{prop}\label{twofacesschu}
Let $A=U_1 \backslash \{ i \}$, $B= U_2 \backslash \{j\}$, $C = C_1 \sqcup C_2$ be three types of first independent sets of the bipartite graph $G^{\pi}$. 
\begin{enumerate}
\item[\normalfont (1)] For any $A, B \in \mathcal{I}^{(1)}_{G^{\pi}}$, $(A,B) \in \mathcal{I}^{(2)}_{G^{\pi}}$.
\item[\normalfont (2)] For any $C, C' \in \mathcal{I}^{(1)}_{G^{\pi}} $, $(C,C') \in \mathcal{I}^{(2)}_{G^{\pi}}$.
\item[\normalfont (3)] $(A,A') \notin \mathcal{I}^{(2)}_{G^{\pi}}$ if and only if there exists a first independent set $U_1 \backslash \{i, i'\} \sqcup C_2$ where $C_2 \subsetneq U_2$ is some vertex set with $|C_2| \leq n-2$.
\item[\normalfont (4)] $(A,C) \notin \mathcal{I}^{(2)}_{G^{\pi}}$ if and only $C_1=\{i\}$ or there exists $C' \in \mathcal{I}^{(1)}_{G^{\pi}}$ with $C_1 \backslash C'_1 = \{i\}$. 
\end{enumerate}
\end{prop}

\begin{proof}
{\normalfont (1)} Suppose that there exist a pair $(A,B) \notin \mathcal{I}^{(2)}_{G^{\pi}}$. Consider the intersection subgraph $\G\{A\} \cap  \G\{B\}$ and assume that it has isolated vertices other than $\{i,j\}$. Consider the isolated vertices in $U_1 \backslash \{i\}$. This means that there exists a two-sided independent set consisting of these isolated vertices and $B$, which is impossible, since $B \in \mathcal{I}^{(1)}_{G^{\pi}}$. Now assume that $\G\{A\} \cap \G\{B\}$ consists of the isolated vertices $\{i,j\}$ and $k \geq 2$ connected bipartite graphs $G_i$. Let the vertex set of $G_i$ consist of $V_i \subsetneq U_1$ and $W_i \subsetneq U_2$. Since $B \in \mathcal{I}^{(1)}_{G^{\pi}}$, there exist an edge $(i, w_i) \in E(G^{\pi})$ for each $i \in [k]$ where $w_i \in W_i$. Symmetrically, since $A \in \mathcal{I}^{(1)}_{G^{\pi}}$, there exist an edge $(j,v_i) \in E(G^{\pi})$ for each $i \in [k]$ where $v_i \in V_i$. However, then for $I \subsetneq [k] $, we obtain the two-sided maximal independent sets of form $\bigsqcup_{i \in I} V_i \sqcup (B \backslash (\bigsqcup_{i \in I} W_i)$ which contradicts the construction of $G^{\pi}$. \\\\
{\normalfont (2)} Let $(x_j, y_j)$ and $(x_{i},y_{i})$ be two essential boxes with $x_{j}> x_i$ and $y_{j} < y_i$, associated to two first independent sets $C$ and $C'$ in $\mathcal{I}^{(1)}_{G^{\pi}}$.  It is enough to check if $\G[C'_1 \sqcup N(C'_1)] \cap \G[C_2 \cap N(C_2)]$ is connected. We observe that the edges of this graph are represented by the square with vertices $(x_{i}+1, y_{j-1}+1)$, $(x_{i} +1, y_{i-1})$, $(x_j, y_{j-1}+1)$, and $(x_j, y_{i-1})$, intersected with the diagram $D(\pi)$. This intersection is also a Ferrer diagram and connected.\\\\
{\normalfont (3)} Consider the intersection subgraph $\G\{A\} \cap \G\{A'\}$. Assume that it has only $\{i, i'\} \subsetneq U_1$ as isolated vertices and $k$ connected bipartite graphs. Then, as in case {\bf 1}, there exist first independent sets $C, C'$ with $C_1 \cap C'_1 = \emptyset$, which is impossible by Lemma \ref{mainlemschu}. Assume that it has the isolated vertices $\{i,i'\} \subsetneq U_1$ and $C_2 \subsetneq U_2$ with $|C_2| \leq n-2$. Then $C:= U_1 \backslash \{i,i'\} \sqcup C_2$ is maximal and thus a first independent set.  \\\\
{\normalfont (4)}  Suppose that $i \in C_1$ and $(A,C) \notin \mathcal{I}^{(2)}_{G^{\pi}}$. Consider the intersection subgraph $\G\{A\} \cap \G\{C\}$. Similarly to last investigations, we conclude $\G[C_1 \sqcup N(C_1)]$ cannot admit $\{i\}$ as its only isolated vertex. If $C_1 = \{i\}$, then the intersection subgraph admits of $|N(C_1)| +1 $ isolated vertices and $\G[C_2 \sqcup N(C_2)]$. Assume that the intersection subgraph consists of the isolated vertex $\{i\} \subsetneq C_1$ and some vertex set $C'_2 \subsetneq N(C_1)$. This means that $C':=C_1 \backslash \{i\} \sqcup C'_2 \sqcup C_2$ is a maximal two-sided independent set. Hence $C' \in \mathcal{I}^{(1)}_{G^{\pi}}$.
\end{proof}
\noindent In order to eliminate the non-rigid cases of $Y_{\pi}$ we introduce the following result. This allows us to focus only on simplicial three dimensional faces of $\sigma_{\pi}$.
\begin{thm}[{\cite[Theorem 3.18]{rigidportakal}}] \label{theorem15}
Let $G \subseteq K_{m,n}$ be a connected bipartite graph. Assume that the edge cone $\sigma_G$ admits a three-dimensional non-simplicial face. Then TV(G) is not rigid.
\end{thm}

\begin{lem}\label{notrigidschus}
Assume that $|\Ess(\pi)| \geq 3$.
\begin{enumerate}
\item[\normalfont (1)] Let $C,C' \in \mathcal{I}^{(1)}_{G^{\pi}}$ with $C'_1 \subsetneq C_1$ and $C_2 \subsetneq C'_2$. If $|C_1| -|C'_1| =1$ and $|C'_2| -|C_2| =1$, then $Y_{\pi}$ is not rigid. 
\item[\normalfont (2)] If there exists a first independent set $C \in \mathcal{I}^{(1)}_{G^{\pi}}$ with $|C_1|=1$ and $|C_2|=n-2$ or  $|C_1|=m-2$ and $|C_2|=1$, then $Y_{\pi}$ is not rigid. 
\end{enumerate}
\end{lem}

\begin{proof}
We refer again to \cite{rigidportakal}. These are the cases from Lemma 3.10 (2)$\casei$ and Lemma 2.11 (2). By Theorem \ref{theorem15}, we conclude that $Y_{\pi}$ is not rigid in these cases. 
\end{proof}

\noindent We denote the image of the first independent sets $A=U_1 \backslash \{i\}$, $B = U_2 \backslash \{j\}$, and $C$ under the map $\Pi$ of Theorem \ref{theorem15} by $\mathfrak{a}=e_i$, $\mathfrak{b}=f_j$, and $\mathfrak{c}$.

\begin{ex}\label{notrigidschuexample}
Let $\pi =[1, 10, 8, 7, 6,9,4,5,2,3] \in S_{10}$  and let us consider the diagram $L(\pi)$. The dotted boxes $L'(\pi)$ form a hook and therefore $Y_{\pi}$ is toric. Consider the first independent sets $C= \{8, 9\} \sqcup \{4, 5, 6, 7, 8\}$ and $C'=\{7, 8, 9\} \sqcup \{5, 6, 7, 8\}$ of the associated connected bipartite graph $G^{\pi} \subsetneq K_{9,8}$. By Lemma \ref{notrigidschus}, $\langle \mathfrak{c},\mathfrak{c}',e_7,f_4 \rangle$ spans a three-dimensional face of $\sigma_{\pi}$ and hence $Y_{\pi}$ is not rigid. 
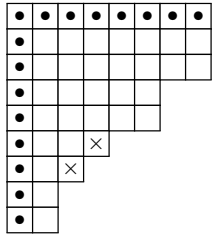
\begin{center}
\ytableausetup{smalltableaux}
\begin{ytableau}
*(white) \bullet & *(white)  \bullet &*(white) \bullet  &*(white) \bullet  & *(white) \bullet  &*(white) \bullet &*(white)\bullet  & *(white)\bullet     \\
*(white)  \bullet & *(white)  &*(white) &*(white) & *(white)  &*(white) &*(white) & *(white)  \\
*(white) \bullet  & *(white)  &*(white) &*(white) & *(white)  &*(white) &*(white) & *(white)    \\
*(white) \bullet & *(white)  &*(white) &*(white) & *(white)  &*(white)  \\
*(white) \bullet & *(white)  &*(white) &*(white)  &*(white) &*(white)\\
*(white)  \bullet & *(white)  &*(white) &*(white) \times \\
*(white)  \bullet & *(white)  &*(white) \times    \\
*(white) \bullet & *(white) \\
*(white) \bullet & *(white) 
\end{ytableau}
\end{center}
\end{ex}
\vspace{0.5cm}
\noindent The cases in Lemma \ref{notrigidschus} are the only cases where $\sigma_{\pi}$ has non-simplicial three-dimensional faces. We conclude this by examining the non 2-face pairs from Proposition \ref{twofacesschu} (3) and (4).\\

\noindent From now on, we may assume that all three-dimensional faces of $G^{\pi}$ are simplicial. In the next proposition, we examine the triples which do not form a three-dimensional face of $\sigma_{\pi}$. 

\begin{prop}\label{threefacesschu}
Let $I$ be a triple of first independent sets of $G^{\pi}$ not forming a three-dimensional face. Assume that any pair of first independent sets of $I$ forms a two-dimensional face. Then the triple $I$ is
\begin{itemize}
\item[\normalfont (1)] $(A,A',A'') \in \mathcal{I}^{(3)}_{G^{\pi}}$ if and only if there exists $C \in \mathcal{I}^{(1)}_{G^{\pi}}$ with $C_1 = U_2 \backslash \{i,i',i''\}$.
\item[\normalfont (2)] $(A,A',C)\in \mathcal{I}^{(3)}_{G^{\pi}}$ if and only if  $C_1=\{i,i'\}$ or there exists $C' \in \mathcal{I}^{(1)}_{G^{\pi}}$ with $C_1 \backslash C'_1 =~ \{i,i'\}$.
\end{itemize}
\end{prop}

\begin{proof}
The first case follows analogously as in the proof of Proposition \ref{twofacesschu} (3). Consider a triple of form $(C,C',C'')$ with $C_1 \subsetneq C'_1 \subsetneq C''_1$ and $C''_2 \subsetneq C'_2 \subsetneq C_2$. Any such triple forms a 3-face, since the intersection graph $\G\{C\} \cap \G\{C'\} \cap \G\{C''\}$ is equal to
 \[\G[C_1 \sqcup N(C_1)] \sqcup \G[(C'_1 \backslash C_1) \sqcup (C_2 \backslash C'_2)] \sqcup \G[(C''_1 \backslash C') \sqcup(C'_2 \backslash C''_2)] \sqcup \G[C''_2 \sqcup N(C''_2)] .\]
 For such triples containing both $A$ and $B$, similar to the arguments in the proof of Proposition \ref{twofacesschu} (1), we conclude that they form 3-faces. Finally, consider the triple $(A,A',C)$. Since $(A,A') \in  \mathcal{I}^{(2)}_{G^{\pi}} $, $i$ and $i'$ cannot be both in $N(C_2)$. Assume that $i \in C_1$ and $i' \in N(C_2)$. Since $(A,C)$ and $(A',C)$ form 2-faces, the triple $(A,A',C)$ forms a 3-face. Hence we have that $\{i,i'\} \subseteq C_1$. The statement follows by the analysis similar to that in the proof of  Proposition \ref{twofacesschu} (4).
\end{proof}

\begin{rem}
In addition to the triple in Proposition \ref{threefacesschu}, the triples of first independent sets of $G^{\pi}$, containing the pairs in Proposition \ref{twofacesschu} (3) and (4) do not form a three-dimensional face of $\sigma_{\pi}$.
\end{rem}
\subsection{Classification of rigid toric varieties $Y_{\pi}$}
\noindent The following two results classify the rigid toric matrix Schubert varieties in terms of edge cone $\sigma_{\pi}$ and in terms of its Rothe diagram $D(\pi)$.

\begin{thm} \label{toricschurigid}
The toric variety $Y_{\pi}=\TV(\sigma_{\pi})$ is rigid if and only if the three-dimensional faces of $\sigma_{\pi}$ are all simplicial.
\end{thm}

\begin{proof}
We have proven the statement for $|\Ess(\pi)| = 1,2$. We prove it now for $|\Ess(\pi)| \geq 3$. We examine the non 2-faces pairs from Proposition \ref{twofacesschu} and non 3-face triples from Proposition \ref{threefacesschu}. \\

\noindent {\bf{1.}} Suppose that $(e_1,e_2,e_3)$ does not span a 3-face and $(e_1,e_2)$, $(e_1,e_3)$ and $(e_2,e_3)$ do span 2-faces. By Proposition \ref{threefacesschu}, there exists a first independent set $C \in \mathcal{I}^{(1)}_{G^{\pi}}$ with $C_1 = U_1 \backslash \{1,2,3\}$ and $|C_2| \leq n-2$. Assume that $\overline{e_1}$, $\overline{e_2}$, and $\overline{e_3}$ are vertices in $Q(R)$ for some deformation degree $R \in M \cong \ZZ^{m+n}/\overline{(1,-1)}$.  Let $\mathfrak{a} \in \sigma_{\pi}^{(1)}$ be an extremal ray. Since $(\mathfrak{a},e_i, e_j)$ spans a 3-face of $\sigma_{\pi}$ for every $i, j \in [3]$ and $i \neq j$, we are left with showing that there exists no such $\overline{\mathfrak{a}} \in Q(R)$. However, even though we have that $R_i \leq 0$, for every $i \in [m+n] \backslash \{1,2,3\}$, $\overline{\mathfrak{c}} \in Q(R)$. \\

\noindent {\bf{2.}} Suppose that $(e_1,e_2,\mathfrak{c})$ does not span a 3-face and $(e_1,e_2)$, $(e_1,\mathfrak{c})$ and $(e_2,\mathfrak{c})$ do span 2-faces. By Proposition \ref{threefacesschu}, $|C_1| = \{1,2\}$ or there exists $C' \in \mathcal{I}^{(1)}_{G^{\pi}} $ such that $C_1 \backslash C'_1 = \{1,2\}$. Assume that $\overline{e_1}$, $\overline{e_2}$, and $\overline{\mathfrak{c}}$ are vertices in $Q(R)$ for some deformation degree $R\in M$. If $|C_1| = \{1,2\}$, then there exists $b \in N(C_1)$ such that $\overline{\mathfrak{b}} \in Q(R)$ is not a lattice vertex or there exist at least three vertices $b_i \in N(C_1)$ such that $\overline{\mathfrak{b_i}}$ is a lattice vertex in $Q(R)$. If $C_1 \backslash C'_1 = \{1,2\}$, then either $\overline{\mathfrak{c}}' \in Q(R)$ or $\overline{\mathfrak{b}} \in Q(R)$ for $b \in C'_2 \backslash C_2$. \\

\noindent {\bf{3.}} Suppose that $(e_1, e_2)$ does not span a 2-face and $\overline{e_1}$ and $\overline{e_2}$ are in $Q(R)$ for some deformation degree $R \in M$. Then there exists a first independent set $C = C_1 \sqcup C_2 \in \mathcal{I}^{(1)}_{G^{\pi}}$ with $C_1 = U_1 \backslash \{1,2\}$ and $2\leq |C_2| \leq n-2$. Remark that for any other two-sided first independent set $C' =C_1' \sqcup C_2' \in \mathcal{I}_{G^{\pi}}^{(1)}$, the pair $(C,C') \in \mathcal{I}_{G^{\pi}}^{(2)}$. Assume that there exist $k$ vertices $\overline{f_{j}}$ in $Q(R)$ where $j \in [k] \subseteq [n]$. If $k=0$, then $\mathfrak{c}$ is a non-lattice vertex in $Q(R)$. If $k=1$, then $\overline{f_{1}}$ is a non-lattice vertex in $Q(R)$. If $k \geq 3$, there can be at most one non 2-face pair say $(f_{1},f_{2})$. However, the other triples of type $(A,B,B')$ not containing both $U_2 \backslash \{1\}$ and $U_2 \backslash \{2\}$ form 3-faces.\\ 

\noindent Suppose now that $\overline{\mathfrak{c}^i} \in Q(R)$ is a lattice vertex. We can assume that there exists only one such extremal ray $\mathfrak{c}^i$, since any triple of type $(C,C',U_1 \backslash \{1\})$ and $(C,C', U_1 \backslash \{2\})$ form 3-faces. Moreover there exists at most one $f_{j'}$ such that $(f_{j'}, \ce^i)$ do not span a two-dimensional face. Hence we obtain that $V(R) / \CC (1,\ldots,1) =0 $ for this deformation degree $R \in M$. It leaves us to check the case where $k=2$. In this case, if the pair $\{f_{1}, f_{2}\}$ do not span a 2-face $\sigma_{\pi}$, then there exists a first independent set $C'' = C''_1 \sqcup C''_2 \in \mathcal{I}_{G^{\pi}}$ with $C''_2 = U_2 \backslash \{1,2\}$ and $|C''_1| \leq m-3$.  Then the only other vertex in $Q(R)$ is $\overline{\ce}$ and it is not a lattice vertex. Furthermore, $(e_i, f_j, \ce)$ spans three-dimensional faces of $\sigma_{\pi}$ for $i \in [2]$ and $j \in [2]$. Last, assume that $(f_{j_1}, f_{j_2})$ spans a 2-face of $\sigma_{G^{\pi}}$. As in the case where $k \geq 3$, it is enough to check the cases for only one vertex $\overline{\ce^i}$ in $Q(R)$. There exists at most one non 2-face pair containing $\ce^i$, say $(f_{j_1}, \ce)$. But then $(\ce_j, f_{j_2},e_1)$ is a 3-face of $\sigma_{G^{\pi}}$.  \\

\noindent {\bf{4.}} Lastly, suppose that $\{\ce, e_i \}$ does not span a 2-face and $\overline{\ce}$ and $\overline{e_i}$ are in $Q(R)$ for some deformation degree $R \in M$. Remark here that we excluded the cases where there exist non-simplicial three-dimensional faces. This means $\ce$ and $e_i$ forms 2-faces with each extremal ray of $\sigma_{\pi}$. Assume that there exist more than three vertices in $Q(R)$ other than $\overline{\ce}$ and $\overline{e_i}$. We examined the cases where non 3-face $(e_1,e_2,e_3)$ appears and where non 2-face $(e_1,e_2)$ appears in $Q(R)$. Therefore we assume that there exists another non 2-face pair, say $(\ce^{*}, e_j)$. But, since $\ce^*$ and $e_j$ also forms 2-faces with each extremal ray of $\sigma_{\pi}$, it is enough to check the cases where there exist less than five vertices in $Q(R)$.\\
Let us first consider the case where there exist exactly two more vertices in $Q(R)$ other than $\overline{\ce}$ and $\overline{e_i}$. We first start with the non 2-face pair $(A,C)$ where $C_1 = \{m\}$ and $A = U_1 \backslash \{m\}$. Then there exists a non-lattice vertex $\overline{j} \in Q(R)$ where $j \in U_2 \backslash C_2$. We observe that there exists no other first independent set $C' \in \mathcal{I}^{(1)}_{G^{\pi}}$ such that $C_2 \subsetneq C'_2$. Therefore it is impossible that there exists another non 2-face pair containing $\overline{\ce'}$.\\

\noindent In the other case where $(\ce,e_i)$ does not span a 2-face, there exists an extremal ray, say $\ce'$ such that $\ce' = e_i + \ce -  \sum_{j \in C'_2 \backslash C_2} f_{j}$. The vertex $\overline{\ce'}$ is in $Q(R)$, unless there exists $\overline{f_j} \in Q(R)$ where $j \in C'_2 \backslash C_2$. This vertex cannot be $\overline{f_j}$ with $\{j\}= C'_2 \backslash C_2$, because then $(\ce,\ce',e_i,f_j)$ spans a 3-face. Hence $\overline{\ce'}$ is one of these two vertices. It remains to check the case where other vertex is $\overline{e_{i-1}}$. Then, there exists a first independent set $C'' \in \mathcal{I}_{G^{\pi}}^{(1)}$. We have that $\overline{\ce''} \notin Q(R)$ if and only if there exists $\overline{f_j'}$ with $j' \in C''_2 \backslash C'_2$, by the same reasoning as before. Lastly, assume that there exists only one lattice vertex in $Q(R)$ other than $\overline{\ce}$ and $\overline{e_i}$.  We observe that $\overline{\ce'}$ is a lattice vertex of $Q(R)$ if there exist some $\overline{f_j} \in Q(R)$ where $j \in C'_2 \backslash C_2$. Therefore we assume that this lattice vertex is $\overline{f_j}$ for some $j \in [n]$. In order to obtain $\langle R, \ce'  \rangle = 0$, we must have $\{j\} = C'_2 \backslash C_2$, but this implies that $(\ce,\ce',e_i,f_j)$ is a 3-face of $\sigma_{\pi}$. 
\end{proof} 

\noindent We interpret the rigidity of $Y_{\pi}$ by giving certain conditions on the Rothe diagram. 

\begin{cor}\label{reformulation}
Let $\Ess(\pi)= \{ (x_i,y_i) \ | \ x_{k+1} < \ldots <x_{1} \text{ and } y_{1} < \ldots < y_{k+1} \}$ with $k \geq 3$. Then the toric variety $Y_{\pi}$ is rigid if and only if 
\begin{itemize}
\item[$\bullet$] $(x_1,y_1) \neq (m,2) $ and $(x_{k+1}, y_{k+1}) \neq (2,n)$ or
\item[$\bullet$] for any $i \in [k], $ $(x_{i}, y_i) \neq (x_{i+1} +1 , y_{i+1}-1) $.
\end{itemize} 
\end{cor}

\begin{proof}
It follows by Lemma \ref{notrigidschus} which characterizes the non-simplicial three-dimensional faces.
\end{proof}

\begin{ex}
In the figure of Example \ref{notrigidschuexample}, consider the essential boxes $(x_2,y_2)$ and $(x_3,y_3)$ which are associated to the first independent sets $C'$ and $C$. We obtain that $(x_2,y_2) = (7,3) =(x_{3} +1 , y_{3} -1) $. Therefore $Y_{\pi}$ is not rigid. On the other hand, we observe the toric variety in Example \ref{rigidexample} is rigid.
\end{ex}

\section*{Acknowledgement}
\noindent This is part of the author’s Ph.D. thesis, written under the supervision of Klaus Altmann and published online at Freie Universit\"at Berlin Library.  The author would like to thank Berlin Mathematical School for the financial support.

\bibliographystyle{spmpsci} 
\bibliography{references}

\section*{Data Availability}
Data sharing not applicable to this article as no datasets were generated or analysed
during the current study.

\end{document}